\documentclass[reqno,11pt]{amsart}
\usepackage{graphicx,  enumitem, amsmath, amssymb,amsthm,hyperref,amscd,xcolor,bbm}
\usepackage{mathabx,manfnt,braket,mathrsfs, comment,mathtools}
\usepackage{tabularx}
\usepackage{bm}

\usepackage{tikz}
\usepackage{tikz-cd}
\usepackage{cancel} 
\usetikzlibrary{arrows, matrix,decorations.pathmorphing}
\usepackage[utf8]{inputenc}
   

\numberwithin{equation}{section}

\newtheorem{thm}[equation]{Theorem}
\newtheorem{prop}[equation]{Proposition}

\newtheorem{lem}[equation]{Lemma}
\newtheorem{cor}[equation]{Corollary}
\theoremstyle{definition}
\newtheorem{defn}[equation]{Definition}
\newtheorem*{rem}{Remark}

\newcommand{\Uu}{\mathscr{U}}
\newcommand{\rat}{\mathrm{Rat}}
\newcommand{\gal}{\mathrm{Gal}}

\newcommand{\norm}[1]{\left\Vert#1\right\Vert}
\newcommand{\abs}[1]{\left\vert#1\right\vert}
\newcommand{\cx}{{\mathbb{C}}}
\newcommand{\rl}{{\mathbb{R}}}

\newcommand{\D}{\mathbb{D}}
\newcommand{\Z}{\mathbb{Z}}
\newcommand{\N}{\mathbb{N}}

\newcommand{\Q}{\mathbb{Q}}
\newcommand{\slb}{{\mathsf{L}(B)}}

\newcommand{\cb}{\mathbb C}
\newcommand{\zb}{\mathbb Z}

\newcommand{\gdot}{\odot}
\newcommand{\one}{\mathbbm{1}}

\newcommand{\hk}{\mathscr{H}_{k}}
\newcommand{\sk}{\mathscr{S}_{k}}

\newcommand{\ol}{\overline}

\newcommand{\dbar}{\ol\partial}
\newcommand{\wt}{\widetilde}

\newcommand{\Identity}{\normalfont\large I}
\newcommand{\rvline}{\hspace*{-\arraycolsep}\vline\hspace*{-\arraycolsep}}

\DeclareMathOperator{\Aut}{Aut} 
 
\DeclareMathOperator{\sgn}{sgn} 
\DeclareMathOperator{\adj}{adj}

\newcommand\ipr[1]{\left\langle #1 \right\rangle}

\title{$L^p$-regularity of the Bergman projection on quotient domains}
\author{Chase Bender}
\address{Department of Mathematics, Central Michigan University, Mt Pleasant, MI 48859, USA.}
\email{bende1cc@cmich.edu}

\author{Debraj Chakrabarti}
\address{Department of Mathematics, Central Michigan University, Mt Pleasant, MI 48859, USA.}
\email{chakr2d@cmich.edu}
\urladdr{http://people.cst.cmich.edu/chakr2d}

\author{Luke  Edholm}
\address{Department of Mathematics, University of Michigan, Ann Arbor, MI 48109, USA.}
\email{edholm@umich.edu}

\author{Meera Mainkar}
\address{Department of Mathematics, Central Michigan University, Mt Pleasant, MI 48859, USA.}
\email{maink1m@cmich.edu }

\thanks{Chase Bender was supported by a Student Research and Creative Endeavors grant from Central Michigan University.\\
Debraj Chakrabarti  was partially supported by  National Science Foundation grant DMS-1600371.}

\begin{document}
\maketitle
\begin{abstract} 
We obtain sharp ranges of $L^p$-boundedness for domains in a wide class of Reinhardt domains representable as sub-level sets of monomials, by expressing them as quotients of simpler domains. We prove a general transformation law relating $L^p$-boundedness on a domain and its quotient by a finite group. The range of $p$ for which the Bergman projection is $L^p$-bounded on our class of Reinhardt domains is found to shrink as the complexity of the domain increases.
\end{abstract}

\section{Introduction}

\subsection{Main result}\label{sec-mainresult} Let  $n\geq 2$, and for each $1\leq j \leq n$,  let $b^j=(b^j_1,\dots, b^j_n)\in\Q^n$ be an $n$-tuple of rational numbers.  
Let $\mathscr{U}\subset \cx^n$ be a bounded domain (open connected subset) of the form
\begin{equation}
\label{eq-udef}
\mathscr{U}=\left\{z\in\cx^n: \text{ for } 1\leq j \leq n,\quad  \prod_{k=1}^n\abs{z_k}^{b^j_k} <1  \right\},
\end{equation}
where it is understood that a point $z\in \cx^n$
does not belong to $\mathscr{U}$, if for
some $1\leq j \leq n$,  the quantity $ \prod_{k=1}^n\abs{z_k}^{b^j_k}$ is not defined due to division by zero. We call a domain such as $\Uu$ a \emph{monomial polyhedron}.

 We refer the reader to  subsection~\ref{sec-u} below for  a discussion of the significance of monomial polyhedra in complex analysis. Our main result is the following:
 
 \begin{thm}
 	\label{thm-main}
 	Suppose that the monomial polyhedron $\mathscr{U}$ of \eqref{eq-udef} is bounded. Then there is a positive integer $\kappa(\mathscr{U})$ (the \emph{complexity} of $\Uu$, whose computation is described below)  such that the Bergman projection 
 	on $\Uu$ is bounded on $L^p(\mathscr{U})$ if and only if
 	\begin{equation}
 	\label{eq-bounds}\frac{2\kappa(\mathscr{U})}{\kappa(\mathscr{U})+1}<p< \frac{2\kappa(\mathscr{U})}{\kappa(\mathscr{U})-1}.
 	\end{equation}
 \end{thm}
 To compute $\kappa(\Uu)$, first define for a vector $x\in \mathbb{Q}^n\setminus\{0\}$ the positive integer 
 $\mathsf{h}(x)$ (the \emph{projective height} of $x$) 
 as follows. If we think of $x$ as the homogeneous coordinates of a point $[x]$ in the rational projective space $\mathbb{P}^{n-1}(\mathbb{Q})$, there is clearly an \emph{integer} vector $y\in \Z^n$ such that 
  $[y]=[x]$ (i.e., there is a $\lambda \in \mathbb{Q}\setminus \{0\}$ such that $y=\lambda x$), and 
 we have additionally that $\gcd(y_1,\dots, y_n)=1$.
  We then set
  \begin{equation}
      \label{eq-height}\mathsf{h}(x)=\sum_{j=1}^n \abs{y_j}.
  \end{equation}
 We can think of $\mathsf{h}$ as a \emph{ height function} on $\mathbb{P}^{n-1}(\mathbb{Q})$ 
 in the sense of Diophantine geometry, uniformly comparable to the standard multiplicative  
 height function (see \cite[pp. 174 ff.]{silverman}).

Let $B$ be the $n\times n$ matrix  whose entry in the $j$-th row and $k$-th column is $b^j_k$, i.e., the $j$-th row of $B$ is the multi-index $b^j$ in \eqref{eq-udef}. It will follow from our work below (Proposition~\ref{prop-uopen})  that the matrix $B\in M_n(\mathbb{Q})$ is invertible. We define
\begin{equation}
\label{eq-kappa}
\kappa(\mathscr{U})= \max_{1\leq k \leq n} \mathsf{h}(B^{-1}e_k),
\end{equation}
where $e_k$ denotes the $n\times 1$ column vector all whose entries are zero, except the $k$-th, which is 1. 
Notice that $B^{-1}e_k$ is simply the  $k$-th \emph{column} of the matrix  $B^{-1}$, 
that is, the arithmetic complexity of  the monomial polyhedron $\mathscr{U}$  is the maximum projective height of the columns of $B^{-1}$, where $B$ is the rational matrix whose rows are the multi-indices occuring in the $n$ inequalities that define $\mathscr{U}$ in \eqref{eq-udef}. It will be shown below in Proposition~\ref{prop-invariance} that the integer $\kappa(\Uu)$ is determined only by the domain $\Uu$ and not 
the particular representation on the right hand side of \eqref{eq-udef}.

\subsection{Singular Reinhardt domains in complex analsysis}
\label{sec-u}
Except in the degenerate case when it reduces to a polydisc (e.g. when $b^j=e_j$, the $j$-th natural basis vector of $\Q^n$), the domain $\Uu$ is a Reinhardt pseudoconvex domain (with center of symmetry at the origin)
such that the origin is a boundary point.  These \emph{ singular Reinhardt domains} (their boundaries are not Lipschitz at 0) display 
pathological holomorphic extension phenomena: the best-known example is that 
the Hartogs triangle $\{\abs{z_1}<\abs{z_2}<1\}\subset \cx^2$, corresponding to a $\Uu$ with $b^1=(1,-1), b^2=(0,1)$
(see \cite{behnke1933,sibony1975}). For example,
on  a singular Reinhardt domain, each holomorphic function smooth up to the closure extends holomorphically to a fixed
neighborhood of the closure (see \cite{sibony}), something which is impossible for smoothly bounded pseudoconvex
domains (\cite{hakimsibony,catlin}).  Therefore, a profound understanding of function theory on these domains
is an important step in extending classical results  on the regularity of the $\dbar$-problem
(and associated operators such as the Bergman projection) to new and more general settings (see \cite{jarpflubook}). Monomial polyhedra are an interesting class of such singular Reinhardt pseudoconvex domains with
tractable geometry and some very interesting properties. They can be compared to analytic polyhedra in the classical theory of pseudoconvex domains, model exhausting domains where 
explicit computations are possible (see \cite[Section 24]{vladi}).

The striking phenomenon  observed in Theorem~\ref{thm-main}  was first noticed (see \cite{EdhMcN16b}) in the setting of 
the so-called {\em generalized Hartogs triangles},  defined for coprime positive integers $k_1,k_2$ as
\begin{equation}\label{eq-genhart}
H_{k_1/k_2}=\{(z_1,z_2)\in \cx^2: \abs{z_1}^{k_1/k_2}<\abs{z_2}<1\},
\end{equation}
which corresponds to $b^1=\left(\frac{k_1}{k_2},-1\right)$ and $b^2=(0,1)$.  
It is striking that the range \eqref{eq-bounds} should depend, not on the shape of the domain as a subset 
of $\cx^2$ (which is determined in the case of $H_{k_1/k_2}$ by the ``fatness exponent"  $\frac{k_1}{k_2}$ ),
but on the complexity (which, for $H_{k_1/k_2}$ is $k_1+k_2$) , the range becoming narrower as the complexity rises. As a limiting case, if $\gamma>0$ is irrational, on the domain $\{\abs{z_1}^\gamma<\abs{z_2}<1\}\subset \cx^2$ (a domain of infinite complexity!), the Bergman projection is bounded in the $L^p$-norm, only if $p=2$
The proofs of the precise range of $L^p$-boundedness of the Bergman projection 
on $H_{k_1/k_2}$  and its generalizations  in earlier
work (\cite{EdhMcN16b,Chen17,huo1,zhang1,zhang2}) consist of an explicit computation of the Bergman kernel, followed by an application of Schur's theorem on $L^p$ boundedness of operators defined by integral kernels to determine the range of $L^p$-boundedness. Other authors have used techniques of classical 
harmonic analysis, such as weak-type endpoint estimates along with interpolation in $L^p$-spaces, Muckenhoupt $A_p$ weights etc. to study related questions.  See \cite{chakzeytuncu, Edh16,EdhMcN16,ChEdMc19, wick,CKY19,EdhMcN20,yuan2} for other results in this circle of ideas.

Theorem~\ref{thm-main} not only 
encompasses  the known examples of domains on which the relation between 
the regularity of the Bergman projection and arithmetic complexity has been observed, but also
substantially extends this class of domains. Its proof is based on an understanding of the geometry of monomial polyhedra as quotient domains.  It is hoped that this will eventually lead to a 
deeper understanding of this mysterious notion of complexity, and its extension to other contexts.

\subsection{Ingredients in the proof of Theorem~\ref{thm-main}} 
The range of $L^p$-boundedness of the Bergman projection on a domain  is a function theoretic property determined by its  Hermitian geometry, but the full extent of this relationship is yet to be understood.  This article brings to bear a new perspective on this problem, in the case of monomial polyhedra: one in which the domain is realized as a \emph{quotient} of a simpler domain under the action of a group of biholomorphic automorphisms $\Gamma$ (Theorem~\ref{thm-geometry} below).  

Our approach to the geometry of $\Uu$ in Section~\ref{sec-geometry} below
is inspired by the observation 
that in the ``log-absolute coordinates" $\xi_k=\log\abs{z_k}$, it is represented
as
\begin{equation}
    \label{eq-logabs}\sum_{k=1}^n b^j_k\xi_k<0,\quad \text{ for each } 1\leq j \leq n,
\end{equation}
which is an \emph{open polyhedral cone} (intersection of open half-spaces) in the sense of convex geometry. By a classical result 
 (see \cite[Theorem~1.3, page~30]{ziegler}, also \cite[Section~3.1]{gb}), such a polyhedral 
cone can also be represented as the cone generated by its extreme points, i.e., it is the  image of an orthant (in a possibly higher dimensional Euclidean space) under a linear map. In Section~\ref{sec-geometry}, we prove an analogous statement for $\Uu$: there is a domain $\D^n_\slb\subset\cx^n$, which is  the product of a certain number of unit discs with a certain number of punctured unit discs, and a proper holomorphic map  $\Phi_A:\D^n_\slb\to \Uu$
which is of the ``quotient type" (see Definition~\ref{def-quotienttype} below), i.e., off some small analytic sets in the source and target, it is essentially a quotient map by a group $\Gamma$ of automorphisms of the source. It turns out that the map $\Phi_A$ is of ``monomial type" in the sense of  \cite{nagelpramanik}, i.e. an $n$-dimensional analog of the branched covering map from the disc to itself given by $z\mapsto z^a$ for an integer $a>0$. This fact has several pleasant consequences and  facilitates  computations.

One of the consequences of the existence of the map $\Phi_A$ is that it allows us to compute the Bergman kernel of $\Uu$ explicitly. This has been a crucial step in the study of the Bergman projection on such domains in all prior investigations.  In this paper, however, we avoid computing Bergman kernels and directly study the transformation properties of $L^p$-Bergman spaces. However, we do show in Proposition~\ref{prop-rationality}, using Theorem~\ref{thm-geometry} that the Bergman kernel of $\Uu$ is a rational function. 

In  Section~\ref{sec-transformation} we study how  $L^p$-Bergman spaces
and the Bergman projection acting on $L^p$-Bergman spaces transform under proper holomorphic maps of quotient type.
This point of view leads to a transformation law (Theorem~\ref{thm-transformation} below) relating the $L^p$-Bergman spaces and Bergman projections of the source and target -- one that is closely connected to the well-known  Bell's transformation law relating the Bergman kernels. One new ingredient here is the use of subspaces \emph{invariant} under the action of the deck-transformation group of the proper holomorphic map, which allows us to state a sharp result which can be used in the Proof of Theorem~\ref{thm-main}.
We believe that the considerations of Section~\ref{sec-transformation}  have an independent interest beyond their application here.

The transformation law Theorem~\ref{thm-transformation} 
can be used to pull back the problem from $\Uu$ to the well-understood domain $\D^n_\slb$, which is the polydisc, except for a missing analytic hypersurface, 
and the pulled-back problem can be solved using classical estimates on the polydisc. This is achieved in Sections~\ref{sec-unbounded} and \ref{sec-bounded}, completing the proof of Theorem~\ref{thm-main}.

The germ of the idea of relating the $L^p$-regularity of the Bergman projection with the properties of a ``uniformizing" map from a simpler domain may already be 
found in  \cite{chakzeytuncu, CKY19}. The sharper version of this technique presented in this paper may be thought of as a step towards a unified understanding of the way in which boundary singularities affect the mapping properties of the Bergman projection.

\subsection{Examples in $\cx^2$} It is not difficult to see that a monomial polyhedron in $\cx^n$ is bounded by $n$ Levi-flat ``faces" (in the log-absolute representation \eqref{eq-logabs}, these faces are linear hyperplanes; see the proof of Proposition~\ref{prop-invariance}).
The generalized Hartogs triangles of \eqref{eq-genhart} are special in that one of the faces is a ``coordinate face", i.e.,
represented by a coordinate hyperplane in log-absolute coordinates. In \eqref{eq-genhart}, this is
$\{\abs{z_2}=1\}$ which corresponds to $\{\xi_2=0\}$ in the log-absolute representation.
In two dimensions, 
 the generic monomial polyhedron has two non-coordinate faces, and can be thought of as an intersections of two generalized Hartogs triangles. The Reinhardt shadows in $\rl^2$ (the image of $z\mapsto (\abs{z_1}, \abs{z_2})$) of three monomial polyhedra in $\cx^2$ can be seen in Figure~\ref{fig-1}. The domains corresponding to $(a), (b),$ and $(c)$ are respectively given by:
 \[ \left\{ \abs{z_1}^4 < \abs{z_2} < \abs{z_1}^{1/3} \right\}, \quad \left\{ \abs{z_1}^{1/2} < \abs{z_2} < \abs{z_1}^{1/4} \right\}, \quad \text{ and }\left\{ \abs{z_1}^{1/2} < \abs{z_2} < 1 \right\}.  \]

\begin{figure}
\begin{tikzpicture}[scale=2.25]
\draw[-{latex}, thick] (0,0) -- (1.5,0) node[above] {$|z_1|$};
\node [below] at (.75,0) {$(a)$};
\draw[-{latex}, thick] (0,0) -- (0,1.5) node[right] {$|z_2|$};
\draw [fill=gray!50] (0,0) to [out=90,in=-126.9] (.125,.5) to [out=53.1,in=-161.6] (1,1) to [out=-104,in=26.6] (.5,.0675) to [out=-153.4, in=0] (0,0);

\draw[-{latex}, thick] (2.5,0) -- (4,0) node[above] {$|z_1|$};
\node [below] at (3.25,0) {$(b)$};
\draw[-{latex}, thick] (2.5,0) -- (2.5,1.5) node[right] {$|z_2|$};
\draw [fill=gray!50] (2.5,0) to [out=90,in=-116.6] (2.5675,.5) to [out=63.4,in=-166] (3.5,1) to [out=206.6,in=45] (2.75,.5) to [out=-135,in=90] (2.5,0);

\draw[-{latex}, thick] (5,0) -- (6.5,0) node[above] {$|z_1|$};
\node [below] at (5.75,0) {$(c)$};
\draw[-{latex}, thick] (5,0) -- (5,1.5) node[right] {$|z_2|$};
\draw [fill=gray!50] (5,0) -- (5,1) -- (6,1) to [out=206.6,in=45] (5.25,.5) to [out=-135,in=90] (5,0);
\end{tikzpicture}
\caption{}
\label{fig-1}
\end{figure}

\subsection{$L^p$-theory of the Bergman projection}
We collect here some general information about the Bergman projection, and set up notation for later use. 
  
Let $\Omega$ be a domain (an open connected set) in $\cx^n$. The \emph{Bergman space} $A^2(\Omega)$ is the Hilbert space of holomorphic functions on $\Omega$ which are square integrable with respect to the Lebesgue measure; see \cite{krantzbergman} for a modern treatment.  The space $A^2(\Omega)$ is  a closed subspace of $L^2(\Omega)$, the  usual Hilbert space of measurable functions square integrable with respect to the Lebesgue measure. The \emph{Bergman projection} is the orthogonal projection 
\[
\bm{B}_{\Omega}:L^2(\Omega)\to A^2(\Omega).
\]

The construction of Bergman spaces has a contravariant functorial character. If $\phi:\Omega_1\to \Omega_2$ is a {suitable} holomorphic map of 
domains, we can associate a continuous linear mapping of Hilbert spaces $\phi^\sharp: L^2(\Omega_2)\to L^2(\Omega_1)$ defined for each $f\in L^2(\Omega_2)$ by
\begin{equation}\label{eq-sharp}
\phi^\sharp(f)= f\circ \phi\cdot \det \phi',
\end{equation}where $\phi'(z):\cx^n\to \cx^n$ is the complex derivative the map $\phi$ at $z\in \Omega_1$.  It is clear that $\phi^\sharp$ restricts to a map  $A^2(\Omega_2) \to A^2(\Omega_1)$.  We will refer to $\phi^\sharp$ as the \emph{pullback} induced
by $\phi$.  It is not difficult to see that if $\phi$ is a biholomorphism, then the pullback
$\phi^\sharp$ is an  isometric isomorphism of 
Hilbert spaces $L^2(\Omega_2)\cong L^2(\Omega_1)$, and restricts to an isometric isomorphism $A^2(\Omega_2) \cong A^2(\Omega_1)$.  This biholomorphic invariance of 
Bergman spaces  can be understood intrinsically by interpreting the Bergman space as a space of top-degree holomorphic forms (see \cite{kobayashi} or \cite[pp. 178 ff.]{krantzbergman}), and the map $\phi^\sharp$ as the pullback map of forms induced by the holomorphic map $\phi$. This invariance can be extended to proper holomorphic mappings via Bell's transformation formula, and  lies at the heart  of classical applications  of Bergman theory to the boundary regularity of holomorphic maps; see \cite{bellduke,belltransactions,difo,bellcat}.

For $0<p<\infty$, define $L^p$-Bergman spaces $A^p(\Omega)$ of  $p$-th power integrable holomorphic functions on $\Omega$.  For $p\ge 1$, these are Banach spaces when equipped with the $L^p$-norm.  An extensive theory of these spaces on the unit disc has been developed, in analogy with the theory of Hardy spaces (cf. \cite{durenbergman,zhubergman}). Unlike the $L^2$-Bergman space, the general $L^p$-Bergman space is not invariantly determined by the complex structure alone, but also depends on the Hermitian structure of the domain as a subset of $\cx^n$.  An important question about these spaces is the boundedness of the Bergman projection in the $L^p$-norm.  After initial results were obtained  for discs and balls (\cite{ZahJud64, rudin}), the problem was studied on various classes of smoothly bounded pseudoconvex domains using estimates on the kernel (e.g. \cite{PhoSte77,McNSte94}). On these domains the Bergman projection is bounded in $L^p$ for $1<p<\infty$.  Many examples have been given which show that there are domains on which the Bergman projection fails to be bounded in $L^p$ for certain $p$.  See \cite{barrett84, SonmezBarrett,KrantzPeloso08,Hed02}, in addition to the singular Reinhardt domains already mentioned in \cite{chakzeytuncu,EdhMcN16,EdhMcN16b,Chen17,CKY19,huo1,wick}. This paper proves sharp $L^p$-regularity results on   a large class of singular Reinhardt domains. 

\subsection{Two examples of Theorem~\ref{thm-main}} We illustrate the application of Theorem~\ref{thm-main} to two   families of domains generalizing the Hartogs triangle to higher dimensions, recapturing the results of \cite{zhang1,zhang2}.

Let $k = (k_1, \ldots , k_n)$ be an $n$-tuple of positive integers. The domain
\begin{equation}\label{eq-hk}
\hk=\left\{z\in \mathbb{D}^n: \abs{z_1}^{k_1} < \prod_{j=2}^n \abs{z_j}^{k_j}\right\},
\end{equation}
was introduced in \cite{pjm} where it was called \emph{elementary Reinhardt domains of signature 1}, and its Bergman kernel was  computed  explicitly. We  see that  $\hk$ is a monomial polyhedron as in  (\ref{eq-udef}), 
since
\begin{equation}\label{eq-hku}\hk=\left\{z\in \mathbb{C}^n: \abs{z_1}^{k_1} \abs{z_2}^{-k_2}\cdots \abs{z_n}^{-k_n}< 1,\text{ and } \abs{z_j} < 1 \text{ for } 2 \leq j \leq n \right\}.
\end{equation}

 The  matrix $B$ whose rows are multi-indices occuring in the inequalities in (\ref{eq-hku}) is then given by

\[B =
\begin{pmatrix}
 k_1 &\rvline & \begin{matrix}   -k_2 & \hdots & -k_n \end{matrix}\\
\hline
\begin{matrix}
  0 \\ \vdots \\ 0
\end{matrix} & \rvline &
  \begin{matrix} &&\\ &\Identity_{n-1} &\\
  &&
  \end{matrix}
 \end{pmatrix}, \text{ so that }  \  B^{-1} = \frac{1}{k_1}
\begin{pmatrix} 
 1 & \rvline & \begin{matrix}   k_2 & \hdots & k_n \end{matrix}\\
\hline
\begin{matrix}
  0 \\ \vdots \\ 0
\end{matrix} & \rvline &
  \begin{matrix} &&\\ & k_1\cdot \Identity_{n-1} &\\
  &&
  \end{matrix}
 
\end{pmatrix},\] 
where $\Identity_{n-1}$ is the identity matrix of size $n-1$. 
The projective height  of the first column of $B^{-1}$ is 1, and for $2\leq j\leq n$ that of the $j$-th column  is $\dfrac{k_1+k_j}{\gcd(k_1,k_j)}$, using \eqref{eq-height}.
Therefore, we get the complexity  of $\hk$  as (see (\ref{eq-kappa})):
\[\displaystyle{\kappa({\hk}) = \max \left\{ 1, \frac{k_1+k_2}{\mathrm{gcd}(k_1, k_2)}, \ldots, \frac{k_1+k_n}{\mathrm{gcd}(k_1, k_n)}  \right\}}= \max_{2\leq j \leq n} \left\{\frac{k_1+k_j}{\mathrm{gcd}(k_1, k_j)}\right\}. \]
Noting that the function $x\mapsto \frac{x}{x-1}$ is decreasing for $x>1$, and the function $x\mapsto \frac{x}{x+1}$ is increasing,  by 
 Theorem \ref{thm-main}, the Bergman projection on $\hk$ is bounded on $L^p(\hk)$ if and only if 
\[\max_{2\leq j \leq n}\, \frac{2(k_1+k_j)}{k_1+k_j+\mathrm{gcd}(k_1, k_j)} < p < \min_{2\leq j \leq n}\, \frac{2(k_1 +k_j)}{k_1+k_j- \mathrm{gcd}(k_1, k_j)},\]
in consonance with the result of \cite{zhang2}.

The second family of the domains which we denote by $\sk$ gives a different type of generalization   of the Hartogs triangle. For an $n$-tuple of positive integers $(k_1,\dots, k_n)$, we define:
\begin{equation}\label{eq-sk}
\sk =  \left\{ z\in \mathbb{C}^n: \abs{z_1}^{k_1}< \abs{z_2}^{k_2}<\dots <\abs{z_n}^{k_n}<1\right\}.
\end{equation}
In \cite{park}, the Bergman kernel of $\mathscr{S}_{k}$  was explicitly computed for $n=3$, and in \cite{Chen17} the special case  $k=(1,1,\dots, 1)$ was considered and it was shown that the Bergman projection is $L^p$-bounded on $\mathscr{S}_{(1,1,\dots,1)}$ if and only if $\frac{2n}{n+1}<p< \frac{2n}{n-1}$. In \cite{zhang1}, the  Bergman kernel of $\sk$ was computed  in general, and the range of $L^p$-boundedness of the Bergman projection was determined. Since $\sk$ is a monomial polyhedron given by
\begin{equation}\label{eq-sku}\sk=\left\{z\in \mathbb{C}^n: \text{ for  } 1 \leq j \leq n-1, \abs{z_j}^{k_j}  \abs{z_{j+1}}^{-k_{j+1}} <1, \text{ and } \abs{z_n}^{k_n} < 1 \right\},
\end{equation}
the  matrix $B$ in the definition of complexity and its inverse $B^{-1}$ are given as below:
\[B =
\begin{pmatrix}
 k_1 & -k_2  &  0 & 0 &  \hdots&  0 \\ 
 0&  \phantom{-} k_2 & -k_3 & 0&  \hdots & 0\\
 \vdots & \phantom{k_2} & \ddots & \ddots & \phantom{0} &\vdots \\
 \vdots & \phantom{k_2} & \phantom{k_2}  & \ddots & \ddots& \vdots \\
0 & 0& \hdots & 0 &  k_{n-1} & - k_n \\
 0 & 0 & 0 & \hdots& 0 & \phantom{-} k_n
 \end{pmatrix}, \quad   \  B^{-1} = \frac{1}{K}
\begin{pmatrix} 
 \ell_1 & \ell_1  &  \ell_1 &   \hdots& \hdots&\ell_1 \\ 
 0&  \ell_2 & \ell_2 &  \hdots & \hdots &  \ell_2\\
0 & 0 & \ell_3 &  \hdots & \hdots& \ell_3 \\
 \vdots & \vdots & \phantom{k_2}  & \ddots & \phantom{k_2} & \vdots \\
 0 & 0 & \hdots& \hdots&  0 & \ell_n

\end{pmatrix},\] 

where $K = \prod_{j=1}^n k_j$ and $\ell_j =\frac{K}{k_j} $ for $1 \leq j \leq n.$ Note that the height 
of the $m$-th column of $B^{-1}$ is
\[ h_m= \frac{\sum_{j=1}^m \ell_j}{\mathrm{gcd}(\ell_1, \hdots, \ell_m)}. \]
Now as $m$ increases the numerator $\sum_{j=1}^m \ell_j$ of $h_m$ increases
and the denominator $\gcd(\ell_1, \ldots, \ell_m)$ decreases, so $h_m$ increases with $m$. Therefore, the
complexity of $\sk$  is 
\[\kappa(\sk) = \max_{1\leq m\leq n} h_m =h_n= \frac{\sum_{j=1}^n \ell_j}{\mathrm{gcd}(\ell_1, \hdots, \ell_n)}  .\]
 Therefore,  Theorem \ref{thm-main} shows that the Bergman projection on $\sk$ is bounded on $L^p(\sk)$ if and only if 
\[ \frac{2 \sum_{j=1}^n \ell_j}{ \sum_{j=1}^n \ell_j + \gcd (\ell_1, \hdots \ell_n)} < p < \frac{2 \sum_{j=1}^n \ell_j}{  \sum_{j=1}^n \ell_j  - \gcd (\ell_1, \hdots \ell_n)},\]
recapturing the main result of \cite{zhang1}.

 \subsection{Acknowledgments} We would like to thank Yuan Yuan for  very interesting discussions with DC and LE about this problem and the results of \cite{CKY19} during the 2019 Midwest Several Complex Variables Conference at Dearborn, MI,  Steven Krantz for comments on the same paper made to LE during a visit to Washington University at St. Louis, MO, in  the same year. DC would like to thank the mathematics department of the University of Michigan, Ann Arbor, for its hospitality in Fall 2019.
 We would also like 
to thank Shuo Zhang for pointing out his work in \cite{zhang1, zhang2},  and Włodzimierz Zwonek for 
pointing out his work in \cite{zwonek99, zwonekhab}. We were unfortunately unaware of these contributions
when earlier versions of this paper were posted on the arxiv preprint server. 

\section{Notation}\label{sec-notation}

\subsection{Elementwise operations on Matrices}  
Let $A$ be an $m\times n$ matrix. We denote the entry of $A$ at the $j$-th row and $k$-th column by $a^j_k$, where $1\leq j \leq m, 1\leq k \leq n$.  We use two kinds of products of 
matrices: one is the standard matrix multiplication, denoted by simple juxtaposition $AB$ or sometimes a  dot for clarity: $A\cdot B$.

We also need a second type of multiplication, the \emph{elementwise} or \emph{Hadamard-Schur} product, in which the product of two $m\times n$ matrices $A=(a^j_k)$
and $B=(b^j_k)$ is the $m\times n$ matrix $C$, for which we have $c^j_k=a^j_kb^j_k$, i.e., the element at a certain position of the product is the product of the corresponding entries of the factors. We denote this by
\begin{equation}\label{eq-gdot}
    C=A\gdot B.
    \end{equation}

It will be important to distinguish between column and row vectors.
We denote the $R$-module of $n\times 1$ column vectors by $R^n$, where $R$ can be one of $\Z,\mathbb{Q},\rl, \cx$. The $R$-module
of $1\times n$ \emph{row} vectors is denoted by $(R^n)^\dagger$. We write the entries of the row
vector $a$ (resp. the column vector $b$) as $(a_1,\dots, a_n)$ (resp. as $(b_1,\dots,b_n)^T$).
We let $\one$  denote a $1\times n$ row vector, all whose entries are 1. The positive integer $n$ will be clear from the context: \begin{equation}
    \label{eq-unitm}
    \one=({1,\dots,1}). \quad \text{ ($n$ ones)}
\end{equation}
Similarly, $\one^T$ is the $n\times 1$ column vector, all whose entries are 1. Notice that these are 
identity elements for the elementwise multiplication of row and column vectors.

For real matrices of the same size $A,B$, the notations
\begin{equation}
    \label{eq-order}
    A \succ B, A \succeq B, A\prec B, A \preceq B
\end{equation}
stand for the natural elementwise order, e.g. $A \succ B$ denotes that $a^j_k> b^j_k.$

\subsection{Vector and matrix powers}
If $\alpha$ is a row vector of size $n$, and $z$ is column vector of the same size, we will 
denote
\[z^\alpha =z^{\alpha_1}\dots z^{\alpha_n}=\prod_{j=1}^n z_j^{\alpha_j},\]
whenever the powers $z_j^{\alpha_j}$ make sense, and where we use the convention 
$0^0=1.$
For example, we could have
 $\alpha\in (\mathbb{Z}^n)^\dagger$ and $z\in \cx^n$ such that for each  $j$ such that $\alpha_j<0$, we have $z_j\not =0$. We then say that $\alpha$ is a \emph{multi-index} and $z^\alpha$ is a \emph{(Laurent)  monomial.}
We also set\begin{equation}
\label{eq-monomial}
\varphi_\alpha(z) = z^\alpha.
\end{equation}
Informally, we think of a monomial function as a ``vector power" of a vector variable. 

We will also use ``matrix powers". Recall from the previous subsection that for an $n\times n$ matrix $A$, we denote the element at the $j$-th row and $k$-th column of $A$ by $a_k^j$. Let $a^j$ denote the $j$-th row of $A$, so that each row can be thought of as a multi-index.
We then define for a column vector $z$ of size $n$,
\begin{equation}
    \label{eq-matrixpower}
    z^A=(z^{a^1},\dots, z^{a^n})^T,
\end{equation}
provided each of the monomials is defined. We will also set
\begin{equation}
    \label{eq-Phi}
   \Phi_A(z)=z^A,  
\end{equation}
which is called a \emph{monomial map}.

For row vectors $\alpha, \beta$, square matrices $A,B$ and column vectors { $z,w$} such that the vector and matrix powers are well-defined,
the following  simple algebraic properties of monomials and monomial mappings can be easily verified. Not unexpectedly, these mirror familiar rules of elementary algebra for exponents.

\begin{subequations}
\begin{equation}
    (z\gdot w)^\alpha =z^\alpha\cdot w^\alpha \label{eq-prop1}
    \end{equation}
    \begin{equation}
    (z\gdot w)^A=z^A\gdot w^A \label{eq-prop2}
    \end{equation}
    \begin{equation}
    z^{\alpha+\beta}=z^\alpha\cdot z^\beta \label{eq-prop3}
    \end{equation}
        \begin{equation}
        z^{A+B}=z^A\gdot z^B\label{eq-prop4}
        \end{equation}
        \begin{align}
    (z^A)^\alpha&=z^{\alpha A}& \text{ i.e.}\quad \varphi_\alpha\circ \Phi_A= \varphi_{\alpha A} \label{eq-prop5}
    \end{align}
    \begin{align}
    (z^A)^B&=z^{BA}& \text{ i.e.}\quad \Phi_B\circ \Phi_A= \Phi_{B A}\label{eq-prop6}
    \end{align}
\end{subequations}
\subsection{Two other maps} 
\begin{enumerate}[wide]
    \item The (elementwise) \emph{exponential} map
    \begin{equation}
        \label{eq-expdef}
        \exp: \cx^n \to (\cx^*)^n, \quad \exp(z)=(e^{z_1},\dots, e^{z_n})^T
    \end{equation}
    can be thought of as the exponential map associated to the abelian Lie group $(\cx^*)^n$. 
   Notice that for $\alpha\in (\Z^n)^\dagger$ and $A\in M_n(\Z)$ we have:
    \begin{equation}
        \label{eq-expprop}\exp(z)^\alpha=e^{\alpha z}\text{ and } \exp(z)^A=\exp(Az).
    \end{equation}
  
    \item  The \emph{multi-radius} map
    \begin{equation}
        \label{eq-rhodef}
         \rho : \cx^n\to \rl^n \quad z\mapsto(\abs{z_1}, \dots , \abs{z_n} )^T
    \end{equation}
    restricts to a surjective group homomorphism $(\cx^*)^n \to (\rl^+)^n$, and satisfies, for $\alpha\in (\Z^n)^\dagger$ and $A\in M_n(\Z)$
    \begin{equation}
    \label{eq-rhoprop}
        \rho(z)^\alpha=\abs{z^\alpha} \text{ and } \rho(z)^A=\rho(z^A).
    \end{equation}
\end{enumerate}

 \section{Geometry of monomial polyhedra}\label{sec-geometry}
The main goal of this section is Theorem~\ref{thm-geometry}, which gives a `` uniformization" of a monomial polyhedron
by a monomial proper holomorphic map. This construction is crucial for 
everything that follows. We also show that the Bergman kernel of a monomial polyhedron is rational, though this fact is not used in the sequel. 

\subsection{The matrix $B$}
The definition of the domain $\Uu$ as given in \eqref{eq-udef} can be succinctly rewritten, using the notation 
introduced in the previous Section~\ref{sec-notation} as
\begin{equation}
    \label{eq-udef2}\Uu=\left\{ z\in \cx^n: \rho(z)^B \text{ is defined  and } \rho(z)^B \prec \one^T\right\}, 
\end{equation}
where $B\in M_n(\Q)$ is the matrix whose $j$-th row is $b^j$. In the next proposition, we show
that the matrix $B$ in the definition of $\Uu$ can always be taken to be an integer matrix which is \emph{monotone} in the sense of Collatz (see \cite[pp. 376 ff.]{collatz}):
\begin{prop}\label{prop-uopen} In the representation \eqref{eq-udef2} of the bounded monomial polyhedron $\Uu$ of Theorem~\ref{thm-main}, we may assume without loss of generality (after switching two rows, if necessary) that
\begin{equation}
    \label{eq-B}B\in M_n(\Z), \quad \det B>0, \quad \text{ and } \quad B^{-1}\succeq 0.
\end{equation}
\end{prop}
\begin{proof} 
Let $B\in M_n(\Q)$ be the matrix whose rows are $b^1,\dots,b^n$, where 
for $1\leq j \leq n$, the vector $b^j\in (\Q^n)^\dagger$, 
$b^j=(b^j_1,\dots, b^j_n)$ is as in \eqref{eq-udef}, which  can be written in 
the form \eqref{eq-udef2} using the notation introduced above. 
Notice that if any one of 
the vectors $b^j$ is 0, then $\Uu$ is empty, since the inequality $\rho(z)^{b^j}<1$ becomes $1<1$.

If $\delta_j>0$ is a common denominator for the rational numbers 
$b^j_1,\dots, b^j_n$ , then notice that for a $z\in \cx^n$, the quantity $\rho(z)^{b^j}$ is defined and less than 1 if and only if $\rho(z)^{\delta_j b^j}$ is defined and less than 1. Therefore, we can assume without loss of 
generality that the matrix $B$ has integer entries. Note also that interchanging 
the rows of $B$ simply corresponds renumbering the equations in \eqref{eq-udef}, 
so we can further assume without loss of generality that $\det B\geq 0$.

Now, assuming that $B\in M_n(\Z)$ and $\det B\geq 0$,
suppose for a contradiction that $\det B=0$. We will  show that $\Uu$
is unbounded, which will contradict the hypothesis of Theorem~\ref{thm-main}. Since $\det B=0$, there is a nonzero 
vector $x\in \rl^n$ such that $Bx=0$. Let $r\in \Uu$ be such that $r \succ 0$ (such an $r$ exists since $\Uu$ is non-empty, open, and Reinhardt). Consider the curve $f : \rl \to \cx^n$ parametrized by
\[ f(t)=r\gdot \exp(tx),\]
where $\exp$ is as in \eqref{eq-expdef}. Observe that
\[f(t)^B=r^B\gdot \exp(tx)^B=r^B\gdot \exp(tBx)= r^B, \]
since $Bx=0$. Therefore,
\[\rho(f(t))^B= \rho(f(t)^B)=r^B. \]
But since $r \in \Uu$, it follows that $ \rho(f(t))^B = r^B\prec \one^T$, so in fact $f$ is a curve in $\Uu$.
As $x \in \rl^n $ is nonzero, some component $x_k$ of $x$ is nonzero, so that the $k$-th component of the curve $f$,  given by $f_k(t)= r_k e^{tx_k}$
is unbounded as $t\to \pm\infty$. Thus $\Uu$ contains the unbounded image of the curve $f$, showing that $\Uu$ is unbounded if $\det B=0$. Since by assumption $\Uu$ is bounded, we have $\det B\neq 0$. So $\det B>0$.

To show that $B^{-1}\succeq 0$, we let $A$ be the adjugate of $B$, i.e. $A = \adj B$, and we show that 
\[ \Phi_A((\D^*)^n) \subset \Uu,\]
where $\Phi_A$ is the monomial map as in \eqref{eq-Phi}, and $(\D^*)^n$ is the product of $n$ punctured unit disks.
Since  $\Uu$ is bounded this implies that $\Phi_A((\D^*)^n)$ is bounded. 
Let $w \in \Phi_A((\D^*)^n)$ and let $z \in (\D^*)^n$ be such that $z^A=w$. Notice that $z \in (\D^*)^n$ is equivalent to $0 \prec \rho(z) \prec \one^T$, and we also have $z^A=w \in (\cx^*)^n$. By applying \eqref{eq-rhoprop} and Cramer's rule, we find
\[\rho(w)^{B} = \left(\rho(z)^A\right)^{B}= \rho(z)^{BA}=\rho(z)^{(\det B) I}=\left(\abs{z_1}^{\det B},\dots, \abs{z_n}^{\det B}\right)^T.\]
As $\det B > 0$ we find that $\rho(w)^B \prec \one^T$, i.e. $w \in \Uu$, yielding the desired inclusion $\Phi_A((\D^*)^n) \subset \Uu$.

 If $A \not\succeq 0$ then there exist some $1 \leq j,k \leq n$ such that $a^j_k < 0$. 
Let $z\in (\D^*)^n$ be such that all its components, except the $k$-th are $1/2$. Then the $j$-th component of $\Phi_A(z)$ is given by
\[ z^{a^j} = {z_k^{a^j_k}}\cdot \frac{1}{2^m},\quad \text{where }\quad
m= \sum_{\substack{1 \leq \ell \leq n \\ \ell \neq k}} a^j_\ell. \]
As $a^j_k < 0$ we see that as $z_k\to 0$, the monomial $z^{a^j}$ is unbounded, showing that $\Phi_A((\D^*)^n)$ is unbounded. Therefore, the boundedness of $\Uu$ guarantees that $A = \adj B \succeq 0$.

To complete the proof note that $A \succeq 0$  implies that $B^{-1}=(\det B)^{-1}\adj B\succeq 0$.
\end{proof}
Due to Proposition~\ref{prop-uopen} we can easily establish a more concrete bound on monomial polyhedra:
\begin{cor}
A bounded monomial polyhedron is contained in the unit polydisc $\D^n$, and its boundary contains the unit torus $\mathbb{T}^n=\{\abs{z_j}=1, 1\leq j\leq n\}$.
\end{cor}
\begin{proof}
For a  bounded monomial polyhedron $\Uu$, let $z\in \Uu$ so that $\rho(z)^B\prec 1$ where $B$ is as in Proposition~\ref{prop-uopen}. Since $B^{-1}\succeq 0$, we have for any $r\in \rl^n$ with 
$0\preceq r \prec 1$ that $r^{B^{-1}}\prec 1$. Therefore we have $(\rho(z)^B)^{B^{-1}}\prec 1$, i.e.
$\rho(z)^{B^{-1}B}=\rho(z)\prec 1$. It follows that $z\in \D^n$. The second assertion follows on noting that in the log absolute representation \eqref{eq-logabs} of $\Uu$, the origin (which corresponds to the unit torus of $\cx^n$) is a boundary point.
\end{proof}

In the definition \eqref{eq-udef2}, it is clear that the matrix $B$ is not unique: permuting the order of inequalities, does not change to domain $\Uu$, and multiplying for each $j$ the row $b^j$ by the same rational $\delta_j>0$ also gives exactly the same $\Uu$. Therefore, for any permutation matrix $P$ and any positive
diagonal matrix $D$ with rational entries, the domains corresponding to $B$ and $DPB$ are the same. The following proposition shows that the complexity $\kappa(\Uu)$ is independent of the matrix $B$ and depends only on $\Uu$.
\begin{prop}\label{prop-invariance}
  The complexity $\kappa(\Uu)$ is independent of the choice of the representing matrix $B$ in \eqref{eq-udef2}.
\end{prop}
\begin{proof} Notice that switching two rows of $B$ does not change the complexity of $\Uu$, we can assume that 
the conditions \eqref{eq-B} hold for the representing matrix.
In the log-absolute coordinates $\xi_k=\log\abs{z_k}$ the domain $\Uu$ is represented by the equations
\eqref{eq-logabs}, which define the open polyhedral cone $\mathscr{C}_\Uu=\{B\xi\prec 0\}\subset\rl^n$. We claim
that for each $1\leq j\leq n$, the hyperplane $H_j=\{ b^j \xi = 0 \}$ determined by the $j$-th row of the matrix
$B$ is a face of $\mathscr{C}_\Uu$, i.e. the intersection $\ol{\mathscr{C}_\Uu}\cap H_j\not=\emptyset$.
Indeed, since by Proposition~\ref{prop-uopen}, $B^{-1}\succ 0$ and $B$ is a linear automorphism of $\Q^n$, it is easy to verify that 
the image $B^{-1}(\mathscr{C}_\Uu)$ is precisely the negative orthant $\{\eta\prec 0\}\subset\rl^n$. But
the negative orthant clearly has faces $\{\eta_j=0\}, j=1,\dots, n$, which corresponds to the faces
$H_j$ of $\mathscr{C}_\Uu$.

Now, if $C\in M_n(\Q)$ is another matrix (satisfying \eqref{eq-B}) such that 
\[\Uu= \left\{ z\in \cx^n: \rho(z)^C \text{ is defined  and } \rho(z)^C \prec \one^T\right\},\]
then the faces of $\mathscr{C}_\Uu$ are also given by $\{ c^j \xi = 0 \}, 1\leq j\leq n$,
so that  there exists a permutation $\sigma \in S_n$ such that 
\[ \{ c^j \xi = 0 \} = H_{\sigma(j)}= \{b^{\sigma(j)} \xi = 0 \} \]
for each $1 \leq j \leq n$. Then the transposes  $(c^j)^T$ and $(b^{\sigma(j)})^T$ are both normal to $H_{\sigma(j)}$, so that there is a $\delta_j \in \Q\setminus\{0\}$
such that $c^j = \delta_j b^{\sigma(j)}$. Therefore, $C=DPB$, where $D$ is the diagonal matrix $\mathrm{diag}(\delta_1,\dots,\delta_n)$ and $P$ is the permutation matrix whose $j$-th column is $e_{\sigma(j)}$.
Since $C^{-1}=B^{-1}P^{-1}D^{-1}$ and we have $C^{-1}\succ 0, B^{-1}\succ 0$ by Proposition~\ref{prop-uopen}, it follows that $D\succ 0$. Therefore,
\begin{align*}
            \max_{1 \leq j \leq n} \mathsf{h}(C^{-1}e_j) &= \max_{1 \leq j \leq n} \mathsf{h}(B^{-1}P^{-1}D^{-1}e_j) = \max_{1 \leq j \leq n} \mathsf{h}(\frac{1}{\delta_j}B^{-1}e_{\sigma(j)}) \\
           &=\max_{1 \leq j \leq n} \mathsf{h}(B^{-1}e_{\sigma(j)})   =\max_{1 \leq j \leq n} \mathsf{h}(B^{-1}e_j).
\end{align*}
This shows that $\kappa(\Uu)$ as in \ref{eq-kappa} is well-defined independent of the matrix $B$.
\end{proof}

\subsection{The Jacobian determinant of $\Phi_A$}
For a subset $E \subseteq \{1 ,\dots , n \}$, define
\begin{equation}
    \cx^n_E := \{z \in \cx^n : z_k \neq 0 \text{ for all } k \in E \} = \prod_{k=1}^n \Omega_k
\end{equation}
where $\Omega_k = \cx$ if $k \not\in E$ and $\Omega_k = \cx^*$ if $k \in E$. 
Then for $P \in M_n(\Q)$ define
\begin{equation}\label{eq-kdef}
\mathsf{K}(P) = \{ 1 \leq k \leq n : p^j_k < 0 \text{ for some } 1 \leq j \leq n \},
 \end{equation}
 that is, $\mathsf{K}(P)$ is the set of indices of those columns of $P$ which have at least one negative entry. 
 Notice that if $P\in M_n(\Z)$ the matrix power $z^P$ is defined for $z\in \cx^n$ if and only if $z\in \cx^n_{\mathsf{K}(P)}$, i.e. in computing $z^P$ we do not raise zero to a negative power. Similarly, for $P\in M_n(\Q)$ and a vector $r\in \rl^n$ with $r_j\geq 0$, the rational power $r^P$ is defined provided $r \in \rl^n_{\mathsf{K}(P)} $ for 
 exactly the same reason.
 
The following  analog of the calculus formula $\frac{d}{dx}x^n=nx^{n-1}$  may be found in  \cite[Lemma~4.2]{nagelduke}, 
and will be needed frequently in the sequel.
\begin{lem}
\label{lem-PhiA'}
For any $A \in M_n (\Z)$ and any $z \in \cx^n_{\mathsf{K}(A)}$ we have
\begin{equation}
     \det\Phi'_A(z) =\det A\cdot{z^{\one A-\one }},\label{eq-det2}
\end{equation}
   where $\Phi_A(z)=z^A$ and $\Phi_A'(z):\cx^n\to \cx^n$ is the 
   complex derivative of $\Phi_A$ at the point $z\in \cx^n_{\mathsf{K}(A)}$.
\end{lem}
\begin{proof} First assume that $z$ is such that $z_k{\neq}0$ for each $k$.
For $1\leq j,k \leq n$, the entry in the $j$-th row and $k$-th column of the complex derivative matrix $\Phi_A'(z)$ is $\dfrac{\partial z^{a^j}}{\partial z_k}$, where recall that $a^j$ denotes the multi-index in $(\Z^n)^\dagger$ whose entries constitute the $j$-th row of $A$. If $a_k^j \neq 0$ then 
\[\frac{\partial z^{a^j}}{\partial z_k}=a_k^j z_k^{a_k^j-1} \prod_{\substack{\ell=1 \\ \ell\neq k}}^n z_\ell^{a_\ell^j} = a_k^j\cdot \frac{z^{a^j}}{z_k}.\]
If $a_k^j=0$, then $\dfrac{\partial z^{a^j}}{\partial z_k}=0$, so in fact the above formula holds for all $j,k$. Now, by the representation of a determinant as an explicit polynomials in the matrix entries,  we find
\begin{align*}
    \det \Phi_A'(z) &=\sum_{\sigma \in S_n} \sgn(\sigma) \prod_{j=1}^n \frac{\partial z^{a^j}}{\partial z_{\sigma(j)}}
    =\sum_{\sigma \in S_n} \sgn(\sigma) \prod_{j=1}^n \left(a_{\sigma(j)}^j \frac{z^{a^j}}{z_{\sigma(j)}}\right)\\
    &= \sum_{\sigma \in S_n} \sgn(\sigma) {\prod_{j=1}^n a_{\sigma(j)}^j}   \cdot \frac{\prod_{j=1}^n z^{a^j}}{\prod_{j=1}^n z_{\sigma(j)}}\\
    &= \det A \cdot \frac{\prod_{j=1}^n z^{a^j}}{z^{\one}}
\end{align*}
where we have used the fact that
$\displaystyle{\prod_{j=1}^n z_{\sigma(j)}=\prod_{j=1}^n z_j=z^{\one}}.$ 
Also,
\[\displaystyle{ \prod_{j=1}^n z^{a^j}=\prod_{j=1}^n\prod_{k=1}^n z_k^{a^j_k}= \prod_{k=1}^n z_k^{\sum_{j=1}^na^j_k}=z^{\one A}. 
}\]
Therefore,
\[ \det \Phi_A'(z)=\det A \cdot \frac{z^{\one A}}{z^{\one}}=\det A\cdot{z^{\one A-\one }},
\]
for $z$ such that $z_k{\neq}0$ for each $k$. The result follows by analytic continuation.
\end{proof}

\subsection{A branched covering of the domain $\Uu$}
In this subsection, we will construct a quotient map under a group action from  a domain with simple geometry (a product of
some copies of discs with some copies of punctured discs) to the domain $\Uu$, and this construction will be 
fundamental in the proof of Theorem~\ref{thm-main}.  We first specify what we mean by a quotient map.

\begin{defn}\label{def-quotienttype}Let $\Omega_1, \Omega_2\subset\cx^n$ be domains, let $\Phi: \Omega_1\to \Omega_2$ be a proper holomorphic mapping. Let $\Gamma\subset \mathrm{Aut}(\Omega_1)$ be a group of biholomorphic automorphisms of $\Omega_1$. We will say that $\Phi$ is of \emph{quotient type
	with group }$\Gamma$
if there exist closed lower-dimensional complex-analytic subvarieties $Z_j\subset \Omega_j, j=1,2$ such that $\Phi$ restricts to a covering map
\[ \Phi: \Omega_1\setminus Z_1 \to \Omega_2\setminus Z_2\]
and
for each $z\in \Omega_2\setminus Z_2$, the action of $\Gamma$ on $\Omega_1$ restricts to a transitive action on the fiber $\Phi^{-1}(z)$. Other names in the literature for such proper maps include \emph{regular} and \emph{Galois} proper maps. The group $\Gamma$ will be referred to as the group of
deck-transformations of $\Phi$ (sometimes called the Galois group). 
\end{defn}
Notice that the restricted map $\Phi:\Omega_1\setminus Z_1 \to \Omega_2\setminus Z_2$ becomes a so called 
\emph{regular} covering map (see \cite[page~135 ff.]{massey}), i.e., the covering map gives rise to a biholomorphism between
$\Omega_2\setminus Z_2$ and the quotient $(\Omega_1\setminus Z_1)/\Gamma$, where it can be shown  that $\Gamma$ acts properly and discontinuously on $ \Omega_1\setminus Z_1$. Further, it follows that $\Gamma$ is in fact the full group of deck transformations of the  covering map $\Phi:\Omega_1\setminus Z_1 \to \Omega_2\setminus Z_2$,
and that this covering map has exactly $\abs{\Gamma}$ sheets. Notice that by analytic continuation, the relation $\Phi\circ\sigma=\Phi$ holds for each $\sigma$ in $\Gamma$ on all of $\Omega_1$.

Now let $B$ be a matrix, let the set $\mathsf{K}(B)\subset \{1,\dots, n\}$ be as in \eqref{eq-kdef}, the set of indices of those columns of $B$ which have at least one negative entry. Define a subset $\mathsf{L}(B) \subset \{1,\dots, n\}$ by 
setting
\begin{equation}
    \mathsf{L}(B)=\{1 \leq \ell \leq n: a^k_\ell \neq 0 \text{ for some } k\in \mathsf{K}(B) \}.
\end{equation}
We define the domain $\D^n_\slb =\D^n \cap \cx^n_\slb $, i.e.
\[ \D^n_\slb=\{ z\in \D^n: z_\ell{\neq}0 \text{ for all } \ell \in \mathsf{L}(B) \}. \]
Notice that $\D^n_\slb$ is the product of some copies of the unit disc with some copies of
the punctured unit disc. 
\begin{thm}
\label{thm-geometry}
Let $\Uu$ be the domain of Theorem~\ref{thm-main} above, and assume (without loss of generality) that $\Uu$ is represented as in \eqref{eq-udef2},
where the matrix $B=(b^j_k)$ satisfies the conditions \eqref{eq-B}, and let
\[A = \adj B. \]
Then  the mononial map $\Phi_A$ of \eqref{eq-Phi} maps $\D^n_\slb$ onto 
$\Uu$, and the map so defined
\begin{equation}\label{eq-PhiA}  
    \Phi_{A}: \D^n_\slb \to \Uu
\end{equation} is  a proper holomorphic map of quotient type with group $\Gamma$
consisting of the automorphisms
	$\sigma_\nu: \D^n_{\mathsf{L}(B)} \to \D^n_{\mathsf{L}(B)}$ given by 
	\begin{equation}
	    \label{eq-sigmanu}
	    \sigma_\nu(z)=  \exp\left(2\pi i A^{-1}\nu\right)\gdot z
	\end{equation}
	for  $\nu \in \Z^n$. Further, the group $\Gamma$ has exactly $\det A$ elements.
\end{thm}
Some related results may be found in \cite{nagelpramanik,zwonek99}.
\begin{proof}
By the conditions \eqref{eq-B}, we have $A =(\det B)B^{-1} \succeq 0$,
so $\Phi_A$ is defined on all of $\D^n_\slb$.
We first claim that the image $\Phi_A(\D^n_\slb)$ is contained in $\cx^n_{\mathsf{K}(B)}$. Indeed, if 
$z\in \Phi_A(\D^n_\slb)$, then $z_\ell{\neq}0$ if $a^k_\ell{\neq}0$ for some $k\in \mathsf{K}(B)$. 
Therefore the $k$-th element of the vector $z^A$, i.e.
\begin{equation}\label{eq-zak}
    z^{a^k}=z_1^{a^k_1}\dots z_n^{a^k_n}
\end{equation}
is nonzero, so $z^A\in \cx^n_{\mathsf{K}(B)}.$ Now notice that 
\[\Uu= \{z\in \cx^n: \rho(z)^B \text{  is defined and  } \rho(z)^B\prec \one^T \}\subset \cx^n_{\mathsf{K}(B)}, \]
since $\rho(z)^B$ (and $z^B$) are defined if and only if $z\in \cx^n_{\mathsf{K}(B)}.$ Therefore, as in the proof of Proposition~\ref{prop-uopen}, we have for each point $z\in \D^n_\slb $, that
\[\left(\rho(z)^A\right)^{B}= \rho(z)^{BA}=\rho(z)^{(\det B) I}\prec \one, \]
since $z\in \D^n$ and $\det B>0$. This shows that 
$\Phi_A : \D^n_\slb \to \Uu$ is a well defined holomorphic map. 

We now show that if we think of $\Phi_A$ as a map from $\cx^n$ to itself (recall that $A\succeq 0$), then we have
\begin{equation}
    \label{eq-preimage}\Phi_A^{-1}(\Uu)\subset \D^n_\slb.
\end{equation}
If $z\in \cx^n$ is such that $\Phi_A(z)=z^A\in \Uu$, then 
$\rho(z^A)^B$ is well defined, and $\rho(z^A)^B\prec \one^T$. The fact that
$\rho(z^A)^B$ is well defined is equivalent to $z^A\in \cx^n_{\mathsf{K}(B)}$, i.e., for each $k\in \mathsf{K}(B)$, we have the entry $z^{a^k}{\neq}0.$ Now from
\eqref{eq-zak} we see that whenever $a^k_\ell{\neq}0$ (so $a^k_\ell\geq 1$ since 
$A\succeq 0$), we must have $z_\ell{\neq}0$,
i.e. $z\in\cx^n_\slb$. The other condition $\rho(z^A)^B\prec \one^T$ on the point $z$ shows that 
\[\rho(z)^{(\det B)I}=\rho(z)^{BA} =\rho(z^A)^B \prec \one^T. \]
Since $\det B>0$, this shows that $\rho(z)\prec \one^T $, i.e. $z\in \D^n$. Together these conditions show that $z\in \D^n_\slb$, proving the desired inclusion of \eqref{eq-preimage}.

We can now show that $\Phi_A:\D^n_\slb\to\Uu$ is a proper holomorphic map.
Let $K \subset \Uu$ be compact. Since the topology induced on $K$ from $\Uu$ is 
the same as that induced from $\cx^n$, it follows that  $K$ is closed in $\cx^n$ (and bounded). As $\Phi_A$ is continuous, $\Phi_A^{-1} (K)$ is closed in $\cx^n$ and in fact $\Phi_A^{-1} (K) \subseteq \D^n_\slb$ by \eqref{eq-preimage}, so that it is bounded as well. Thus $\Phi_A^{-1} (K)$ is compact for every compact $K \subset \Uu$, which is to say $\Phi_A$ is proper. 

Now we want to verify that the proper holomorphic map $\Phi_A:\D^n_\slb \to \Uu$ is of quotient type
in the sense of Definition~\ref{def-quotienttype} with group
$\Gamma= \{\sigma_\nu: \nu\in \Z^n \},$
with $\sigma_\nu$ as in \eqref{eq-sigmanu} above. Let 
\[Z_1=Z_2= \{z^\one =0\} =\bigcup_{j=1}^n\{z\in \cx^n: z_j=0\}\]
be the union of the coordinate hyperplanes, which is an analytic variety in $\cx^n$ of codimension one.
Notice that the formula \eqref{eq-det2} shows that the set of regular points of the map $\Phi_A$ (i.e.
the set of points $\{z\in \D^n_\slb: \det\Phi_A'(z){\neq}0 \}$) contains $\D^n_\slb\setminus Z_1$, 
and from the fact that $\Phi_A(z)=z^A$, we see that the regular values of $\Phi_A$ contains the open set $\Uu\setminus Z_2$. Notice that both $\D^n_\slb\setminus Z_1$ and $\Uu\setminus Z_2$ are open subsets of the ``algebraic torus" $(\cx^*)^n$, and
it is clear that $\Phi_A$ maps $(\cx^*)^n$ into itself (and is even a group homomorphism, if $(\cx^*)^n$ is given the group structure from the elementwise multiplication operation $\gdot$). Since a proper holomorphic map
is a covering map restricted to its regular points,  it now follows that 
\[\Phi_A: \D^n_\slb\setminus Z_1\to \Uu\setminus Z_2 \]
is a holomorphic covering map. We need to show that it is regular. 

Let $w\in \Uu\setminus Z_2$ have polar representation $w=\rho(w)\gdot \exp(i\theta)$, where $\theta\in \rl^n$.
Let $z\in \D^n_\slb$ be such that $\Phi_A(z)=z^A=w$. Comparing the radial and angular parts, we see that
\[z=\rho(w)^{A^{-1}}\gdot \exp\left( i A^{-1}\theta\right). \]
Now notice that the angular vector $\theta$ is known only up to an additive ambiguity of $2\pi \Z^n$, i.e., two values of $\theta$ corresponding to the same $w$ differ by $2\pi \nu$ for some integer vector $\nu\in \Z^n$. Therefore, two 
different preimages $z,z'$ of the point $w$ (correspoding to the choices $\theta$ and $\theta+2\pi\nu$ of the angular vector in the polar representation of $w$) are related by 
\begin{align*}
   z' &= \rho(w)^{A^{-1}}\gdot \exp\left( i A^{-1}(\theta+2\pi\nu)\right) 
   = \rho(w)^{A^{-1}}\gdot \exp\left( i A^{-1}\theta\right)\gdot \exp\left(2\pi i A^{-1}\nu\right)\\
   &= \exp\left(2\pi i A^{-1}\nu\right)\gdot z\\
   &=\sigma_\nu(z).
\end{align*}
which shows that the group $\Gamma$ acts transitively on the fiber $\Phi^{-1}(w)$.

Now we want to show that the  group $\Gamma$ has exactly $\det A$ elements.
Consider the map $\psi: \Z^n\to\Gamma$ given by $\psi(\nu)= \sigma_\nu. $
For any $\nu, \mu \in \Z^n$ and any $z \in \D^n_\slb$ we find that 
\begin{equation*}
\begin{split}
    \psi(\nu+\mu)(z)=\sigma_{\nu + \mu}(z) &=\exp \left(2 \pi i A^{-1}(\nu + \mu) \right) \gdot z \\
    &= \exp \left(2 \pi i A^{-1}\nu \right)\gdot \left( \exp\left(2 \pi i A^{-1} \mu \right) \gdot z \right) \\
    &= \sigma_\nu \circ \sigma_\mu (z)\\
    &=\psi(\nu) \circ \psi(\mu) (z).
\end{split}
\end{equation*}
Hence $\psi : \Z^n \to \Gamma$ is a  surjective group homomorphism, and so we have an isomorphism
\[\ol{\psi}: \Z^n/\ker\psi\to \Gamma. \]
given by $\psi([\nu])=\sigma_\nu$ where $[\nu] \in \zb^n/ \ker \psi$ is the class of $\nu \in \zb^n$. Notice that $\nu \in \ker \psi$ if and only if  $\exp(2 \pi i A^{-1} \nu) = \one^T$. This means that 
$A^{-1}\nu\in \Z^n$, so it follows that $\ker \psi = A(\zb^n)$. Therefore, the
	group $\Gamma$ is isomorphic to the quotient $\Z^n/A(\Z^n)$ by the isomorphism $ \ol{\psi}: \Z^n/A(\Z^n)\to \Gamma$
	
To complete the proof it is sufficient to recall that for any matrix $A\in M_n(\Z)$ with $\det A{\neq}0$ we have 
\begin{equation}
    \label{eq-card}
    \abs{\Z^n/A(\Z^n)}=\abs{\det A}.
\end{equation}
For completeness we give two proofs of \eqref{eq-card},  neither of which unfortunately  avoids high technology. The first (see \cite{nagelpramanik} and \cite[Theorem~2.1]{zwonek99}) is based on the Smith Canonical Form 
of an integer matrix (see \cite[pp. 53 ff.]{munkres}): for any $A \in M_n (\zb)$ there exist $P,Q \in GL_n(\zb)$ and a diagonal matrix $D \in M_n(\zb)$, say $D = \mathrm{diag}(\delta_1 ,\dots ,\delta_n)$, such that $A = PDQ$. 
Consider the automorphism of $\Z^n$ given by $x\mapsto P^{-1}x$. This maps $A(\Z^n)$ to $D(\Z^n)$ isomorphically, and therefore leads to an isomorphism $\Z^n/A(\Z^n)$ with $\Z^n/D(\Z^n)$, but $\Z^n/D(\Z^n)$ is isomorphic to the product abelian group $\prod_{j=1}^n\Z/\delta_j\Z$, which has exactly 
$\prod_{j=1}^n\abs{\delta_j}$ elements. On the other hand note that 
\[\abs{\det A}=\abs{ \det(PDQ)}= \abs{\det P\det D\det Q}= \abs{\pm 1\cdot \prod_{j=1}^n \delta_j \cdot \pm 1},\]
which completes the first proof.

The second proof of \eqref{eq-card} is geometric, and based on noticing that the inclusion of abelian groups $A(\Z^n)\hookrightarrow \Z^n$ gives rise to a covering map of  tori
\begin{equation}
    \label{eq-toricover}
    p:\rl^n/A(\Z^n)\to \rl^n/\Z^n
\end{equation}
given by $x+A(\Z^n)\mapsto x+\Z^n$. If we endow $\rl^n$ with the Euclidean metric (thought of as a Riemannian metric), we obtain quotient (flat) Riemannian structures
on both $\rl^n/A(\Z^n)$ and $\rl^n/\Z^n$, so that the covering map \eqref{eq-toricover} is actually a local isometry. Therefore, the number of sheets of this covering map coincides with the ratio $\mathrm{vol}(\rl^n/A(\Z^n))/\mathrm{vol}( \rl^n/\Z^n)$. But we know that $\mathrm{vol}( \rl^n/\Z^n)$ is 1, being also the volume of the unit parallelepiped, and similarly
$\mathrm{vol}(\rl^n/A(\Z^n))$ is the volume of the image of the unit parallelepiped under the action of $A$, which is therefore $\abs{\det A}$. On the other hand, by covering space theory, we see that the number of sheets is equal to the order of the group of deck transformations of the 
covering $p$, which is in turn isomorphic to $\pi_1(\rl^n/\Z^n)/p_*\pi_1(\rl^n/A(\Z^n))=\Z^n/A(\Z^n)$.\end{proof}

\subsection{Rationality of the Bergman kernel of $\Uu$}
 For a domain $\Omega\subset \cx^n$, denote by $\rat\left(\Omega\right)$
 the rational functions of $\cx^n$ restricted to $\Omega$.
 An element $f$ of $\Omega$ is defined on $\Omega\setminus Z_f$, where $Z_f$ is an affine algebraic variety in $\cx^n$.
Algebraically,
$\rat\left(\Omega\right)$ is a field, and is isomorphic to the field of rational functions
$\cx(z_1,\dots,z_n)$ in $n$ indeterminates. The map $\Phi_A:\D^n_\slb\to\Uu$ of \eqref{eq-PhiA} induces a mapping of fields
\[ \Phi_A^*: \rat\left(\Uu\right)\to \rat\left(\D^n_\slb\right)\]
given by $\Phi_A^*(f)= f\circ \Phi_A$.
Denote by $k$ the image 
\begin{equation} \label{eq-kfield}
k=\Phi_A^*(\rat\left(\Uu\right))\subset\rat\left(\D^n_\slb\right).
\end{equation}
Then $k$ is 
clearly a subfield of $\rat\left(\D^n_\slb\right)$. The following lemma records a special case of a very 
general phenomenon:
\begin{lem}
\label{lem-galois} The field extension $\rat\left(\D^n_\slb\right)/k$ is  Galois, and the map
\begin{equation}
    \label{eq-mapj}
\Theta:\Gamma\to \gal\left(\rat\left(\D^n_\slb\right)/k\right), \quad \Theta(\sigma)(f)=f\circ \sigma^{-1} 
\end{equation}
 where $f\in \rat\left(\D^n_\slb\right), \sigma\in \Gamma $, is an isomorphism of the group of deck transformations of the proper map $\Phi_A$ with the 
 Galois group of the extension. 
\end{lem}
\begin{proof}
Denote by $w_j$ the $j$-th component of the map $\Phi_A$, where $1\leq j \leq n$. Notice that then $w_j$ can be 
identified with the function $\Phi_A^*(\pi_j)=\pi_j\circ \Phi_A$ where $\pi_j$ is the $j$-th coordinate function on $\Uu$. Since $\rat\left(\Uu\right)=\cx(\pi_1,\dots, \pi_n)$, the field $k$ is generated over $\cx$ by the functions 
$w_1,\dots, w_n$, i.e. $k=\cx(w_1,\dots, w_n)$. Denoting the coordinate functions on 
$\D^n_\slb$ by $z_1,\dots, z_n$, we see that $\rat\left(\D^n_\slb\right)= k(z_1,\dots, z_n)$. To 
show that $\rat\left(\D^n_\slb\right)/k$ is Galois, we need to show that each $z_j$ is algebraic over $k$, and all the conjugates of $z_j$ over $k$ are already present in $\rat(\D^n_\slb).$

Notice that by definition for each $1\leq j \leq n$, we have
$w_j=z_1^{a^j_1}\dots z_n^{a^j_n}=z^{a^j}$, which follows from the fact that 
$w=\Phi_A(z)=z^A$. Therefore we have
\[z^{\det A\cdot I} = z^{\adj A\cdot A} =(z^A)^{\adj A} =w^{\adj A}. \]
Denoting the $j$-th row of the integer matrix $\adj A$ by $c^j$ we see that
\[z_j^{\det A} = w^{c^j}\in k,\]
so that $z_j$ is a root of the polynomial $t^{\det A}-w^{c^j}\in k[t]$. Further,
all the roots of this polynomial (which are of the form $\omega\cdot z_j $ for a 
$(\det A)$-th root of unity, $\omega\in \cx$) are clearly in $\rat\left(\D^n_\slb\right)$. It follows that 
$\rat\left(\D^n_\slb\right)$ is Galois over the subfield $k$.

For each $\sigma\in \Gamma$ it is clear that  $\Theta(\sigma)$ is an automorphism of the field $\rat\left(\D^n_\slb\right)$. We need to show that it fixes $k$. But for $h=\Phi_A^*g\in k$, where $g\in \rat\left(\Uu\right)$, we have
\[ \Theta(\sigma)h=h\circ\sigma^{-1}=g\circ \Phi_A\circ \sigma^{-1}=g\circ \Phi_A=h,\]
since $\sigma$ is a deck transformation of $\Phi_A$. It follows that $\Theta(\sigma)\in \gal\left(\rat\left(\D^n_\slb\right)/k\right)$. It is easily verified by direct computation that $\Theta$ is an injective group homomorphism. To show that $\Theta$ is surjective,  let $\gamma\in \gal\left(\rat\left(\D^n_\slb\right)/k\right)$, and we have to find a deck-transformation $\theta \in \Gamma$ such that $\gamma(f)=f\circ \theta$, so that $\gamma=\Theta( \theta^{-1})$.
Let $\theta_j\in \rat\left(\D^n_\slb\right)$ be the function $\theta_j=\gamma(z_j)$, where $z_j\in \rat\left(\D^n_\slb\right)$ is the $j$-th coordinate function.
As we saw above, then we have $\theta_j=\omega_j\cdot z_j$, where $\omega_j$ is a $(\det A)$-th root of unity. Then $\theta=(\theta_1,\dots,\theta_n)$ defines a diagonal unitary mapping of $\cx^n$
and therefore maps $\D^n_\slb$ (a product of discs and punctured discs) into itself. Also,  the 
$j$-th component of $\Phi_A\circ \theta$ is given 
by 
\[\theta^{a^j}=\theta_1^{a^j_1}\dots \theta_n^{a^j_n}= \gamma(z_1)^{a^j_1}\dots \gamma(z_n)^{a^j_n}=\gamma(z^{a^j})=\gamma(w_j)=w_j,
\]
which is precisely the $j$-th component of $\Phi_A$.
Therefore, $\Phi_A\circ \theta=\Phi_A$, and $\theta\in \Gamma$. Finally, for a rational function
$f=f(z_1,\dots,z_n)\in \rat\left(\D^n_\slb\right)$, we have
\[\gamma(f)=f(\gamma(z_1),\dots, \gamma(z_n))=f(\theta_1,\dots,\theta_n)=f\circ\theta. \]
\end{proof}
\begin{prop}\label{prop-rationality}
The Bergman kernel of $\Uu$ is a rational function.
\end{prop}
\begin{proof}Choosing a local continuous branch $\theta_j$ of $\arg w_j$ for each coordinate 
of $w\in \Uu$, we can write $w=\rho(w)\gdot \exp(i\theta)$ for $\theta\in \rl^n$. Then 
$\psi(w)= w^{A^{-1}}=\rho(w)^{A^{-1}}\gdot\exp(iA^{-1}\theta)$ is a 
 local inverse to the proper holomorphic map $\Phi_A:\D^n_\slb\to \Uu$.
 Since $\Phi_A$ is of the quotient type,
it follows that  the local branches of $\Phi^{-1}_A$ are $ \{\sigma\circ \psi: \sigma\in \Gamma \}.$
 By Theorem~\ref{thm-geometry}, the map $\sigma:\D^n_\slb\to \D^n_\slb$ is the restriction of a unitary linear map on $\cx^n$, so, noting that 
the argument behind the proof of \eqref{eq-det2} also applies for (locally defined)
rational matrix powers, we have
\begin{align*}
    \det (\sigma\circ \psi)'(w)&= \det\sigma'\cdot \det \psi'(w)=\det \sigma\cdot\det A^{-1}\cdot w^{\one A^{-1}-\one}\\
    &=\det \sigma\cdot\det A^{-1}\cdot(w^{A^{-1}})^\one\cdot \frac{1}{w^\one}= \frac{\det \sigma\cdot\psi(w)^\one}{\det A\cdot w^\one}.
\end{align*}

Therefore, by the Bell transformation formula, for $z\in \D^n_\slb$  and $w\in \Uu$, the Bergman kernels $K_\Uu$ and 
$K_{\D^n_\slb}=K_{\D^n}$ are related by 
\begin{align*}
    \det\Phi'_A(z)\cdot K_{\Uu}\left(\Phi_A(z),w\right)&= \sum_{\sigma\in \Gamma} K_{\D^n}\left(z, \sigma\circ\psi(w)\right)\cdot \ol{\det (\sigma\circ \psi)'(w)}\\
    &= \frac{1}{\det A\cdot \ol{w^\one}}\cdot  \sum_{\sigma\in \Gamma} K_{\D^n}\left(z, \sigma\circ\psi(w)\right)\cdot  \ol{\det \sigma\cdot\psi(w)^\one}
\end{align*}
For a fixed $z\in \D^n_\slb$, consider the function $L$ on $\D^n_\slb$ given by 
\[ L(\zeta)=\sum_{\sigma\in \Gamma} K_{\D^n}\left(\ol{z}, {\sigma(\zeta)}\right)\cdot  \det \sigma\cdot\zeta^\one, \]
where $\ol{z}=(\ol{z_1},\dots, \ol{z_n})\in \D^n_\slb $. Recalling that
	\begin{equation}
	    \label{eq-bpdsk}
	     K_{\D^n}(z,w)= \frac{1}{\pi^n}\prod_{j=1}^n \frac{1}{(1-z_j\ol{w_j})^2},
	\end{equation}
we see that $L\in \rat(\D^n_\slb)$, and 
\[  K_{\Uu}\left(z^A,w\right)= \frac{1}{(\det A)^2\cdot z^{\one A-\one}\cdot\ol{w^\one} }\cdot \ol{L(\psi(w))}.\]
We now claim that $L\in k$, where $k\subset \rat(\D^n_\slb)$ is as in \eqref{eq-kfield}. Since the extension
$\rat(\D^n_\slb)/k$ is Galois, it suffices to show that for each field-automorphism $\tau \in \gal(\rat(\D^n_\slb)/k)$, 
we have $\tau(L)=L$. But by Lemma~\ref{lem-galois}, for each $\tau\in \gal(\rat(\D^n_\slb)/k)$ there is a $\theta\in \Gamma$ such that $\tau(L)=L\circ \theta$. Notice that thanks to \eqref{eq-sigmanu}, we can think of the map $\theta:\D^n_\slb\to \D^n_\slb$ as a linear map 
of $\cx^n$  represented by a diagonal matrix. Therefore we have $(\theta(\zeta))^\one=\det \theta\cdot\zeta^\one. $ Consequently
\begin{align*}
    \tau(L)(\zeta)=L(\theta(\zeta))&= \sum_{\sigma\in \Gamma} K_{\D^n}\left(\ol{z}, {\sigma(\theta(\zeta))}\right)\cdot  \det \sigma\cdot(\theta(\zeta))^\one\\
     &=\sum_{\sigma\in \Gamma} K_{\D^n}\left(\ol{z}, {(\sigma\circ\theta)(\zeta)}\right)\cdot  \det \sigma\cdot\det\theta\cdot\zeta^\one\\
     &=\sum_{\sigma\in \Gamma} K_{\D^n}\left(\ol{z}, {(\sigma\circ\theta)(\zeta)}\right)\cdot  \det (\sigma\circ\theta)\cdot\zeta^\one\\
     &=L(\zeta),
\end{align*}
establishing the claim. 

Now since $L\in k$, there is a function $R(z,\cdot)\in \rat(\Uu)$ such that $L=\Phi_A^* R(z,\cdot)$, i,e,
\[ L(\zeta)= R(z, \Phi_A(\zeta))= R(z, \zeta^A).\]
Therefore we have
\[  K_{\Uu}\left(z^A,w\right)= \frac{1}{(\det A)^2\cdot z^{\one A-\one}\cdot\ol{w^\one} }\cdot \ol{R(z,\Phi_A(\psi(w))}=\frac{1}{(\det A)^2\cdot z^{\one A-\one}\cdot\ol{w^\one} }\cdot \ol{R(z,w)},\]
which shows that for each fixed $z\in \D^n_\slb$, the function $K_{\Uu}\left(z,\cdot\right)$ is rational. 
By the Reinhardt symmetry of $K_\Uu$, there is a function $\wt{K}$ such that 
\[ K_{\Uu}\left(z,w\right)=\wt{K}(z_1\ol{w_1},\dots, z_n\ol{w_n})= \wt{K}(z\gdot \ol{w}).\]
Since for a fixed $z$, the function $w\mapsto\wt{K}(z\gdot \ol{w})$ is rational, it follows 
that the function $t\mapsto \wt{K}(t)=\wt{K}(t_1,\dots, t_n)$ is also rational. Therefore
the function $K_\Uu$ is rational on $\Uu\times \Uu$.
\end{proof}

 \section{Transformation of \texorpdfstring{$L^p$}{Lp}-Bergman spaces under quotient maps}
\label{sec-transformation}
\subsection{Definitions} In order to state our results, we introduce some terminology:

\begin{defn}  We say that a linear map $T$ between Banach spaces $(E_1, \norm{\cdot}_1)$  and $(E_2, \norm{\cdot}_2)$  is a \emph{homothetic isomorphism} 
		if it is a continuous bijection (and therefore has a continuous inverse) and
		there is a constant $C>0$ such that for each $x\in E_1$ we have 
		\[ \norm{Tx}_2 = C \norm{x}_1.\]
\end{defn}		
		
\begin{rem}		
If $T$ is a homothetic isomorphism between Banach spaces, then   $\tfrac{1}{\norm{T}}\cdot T$ is an isometric isomorphism of Banach spaces, so a homothetic isomorphism is simply the product of an isometric isomorphism and a scalar operator. In particular, a homothetic isomorphism between Hilbert spaces preserves angles, and in particular orthogonality of vectors. 
\end{rem}
		
		The following definition and facts are standard:
		 \begin{defn}
		  Let $\Omega\subset\cx^n$ be a domain, let $\lambda>0$ be a continuous function (the weight), and let $0<p<\infty$. Then we define
		 \[ L^p(\Omega, \lambda)= \left\{f:\Omega \to \cx \text{ measurable }: \int_\Omega \abs{f}^p \lambda\, dV < \infty\right\}\]
		 and
		 \[ A^p(\Omega, \lambda)= \left\{f:\Omega \to \cx \text{ holomorphic }: \int_\Omega \abs{f}^p\lambda\, dV < \infty\right\},\]
		 where the latter is called a \emph{weighted Bergman space}. If $p\geq 1$, $L^p(\Omega,\lambda)$ and $A^p(\Omega, \lambda)$ are 
		 Banach spaces with the natural \emph{weighted norm}, and $A^p(\Omega,\lambda)$ is a closed subspace of $L^p(\Omega,\lambda)$.
	\end{defn}

We will make extensive use of the following notion:
\begin{defn}
Given a group $G$ of biholomorphic automorphisms of a domain $\Omega\subset \cx^n$, and a space $\mathfrak{F}$ of functions on $\Omega$, we denote by $\left[ \mathfrak{F}\right]^G$ the subspace of $\mathfrak{F}$ consisting of functions which are $G$-invariant in the following sense
\begin{equation}
    \label{eq-gammainv}
   \left[\mathfrak{F}\right]^G = \{f\in \mathfrak{F} : f=  \sigma^\sharp(f) \text{ for all } \sigma \in G\},  
\end{equation}
where $\sigma^\sharp$ is the pullback induced by $\sigma$ as in \eqref{eq-sharp}.
\end{defn}

\begin{rem}
Interpreting $\mathfrak{F}$ as a space of holomorphic forms on $\Omega$ by associating $f\in \mathfrak{F}$ with the form $fdz_1\wedge \dots \wedge dz_n$, this simply says that the forms in $[\mathfrak{F}]^G$ are invariant under pullback by elements of $G$. 
\end{rem}

\subsection{Transformation of Bergman spaces} With the above definitions, we are ready to state and prove the following elementary facts. For completeness, we give details of the proofs.
	\begin{prop}
	\label{prop-homothety} 
	Let $\Omega_1,\Omega_2$ be domains in $\cx^n$, and let $\Phi:\Omega_1\to \Omega_2$ be a proper holomorphic map of quotient type with group $\Gamma\subset \mathrm{Aut}(\Omega_1)$. Then for $1<p<\infty$, the pullback map $\Phi^\sharp$ gives rise to a homothetic isomorphism
\begin{equation}\label{eq-phisharp1}
\Phi^\sharp: L^p(\Omega_2)\to \left[L^p\left(\Omega_1, \abs{\det \Phi'}^{2-p}\right)\right]^{\Gamma}.
\end{equation}
This restricts to a homothetic isomorphism
\begin{equation}
\label{eq-phisharp2}
\Phi^\sharp: A^p(\Omega_2)\to \left[A^p\left(\Omega_1, \abs{\det \Phi'}^{2-p}\right)\right]^{\Gamma}.
\end{equation}
		
	\end{prop}
\begin{proof}
Let $f$ be a function on $\Omega_2$, and let $g=\Phi^\sharp f$ be its pullback to $\Omega_1$, then we have for each $\sigma \in \Gamma$ that
\[ 
\sigma^\sharp(g)=\sigma^\sharp (\Phi^\sharp f) = (\Phi\circ \sigma)^\sharp f = \Phi^\sharp f =g,  
\]
where we have used the contravariance of the pullback $(\Phi\circ \sigma)^\sharp=\sigma^\sharp \circ\Phi^\sharp$ and the fact that $\Phi\circ \sigma =\Phi$ which follows since the action of $\Gamma$ on $\Omega_2$
restricts to actions on each of the fibers.  This shows that the range of $\Phi^\sharp$ consists of $\Gamma$-invariant functions. Special cases of this invariance were already noticed  \cite{misra,CKY19}.

To complete the proof  of \eqref{eq-phisharp1} we must show that 
\begin{itemize}
\item[(i)]for each $f\in L^p(\Omega_2)$,
\begin{equation}
    \label{eq-phisharphomo}\norm{\Phi^\sharp f}^p_{L^p\left(\Omega_1, \abs{\det \Phi'}^{2-p}\right)}= \abs{\Gamma}\cdot \norm{f}^p_{L^p(\Omega_2)},
\end{equation}
\item[(ii)] The image  $\Phi^\sharp(L^p(\Omega_2))$ is precisely $\left[L^p\left(\Omega_1, \abs{\det \Phi'}^{2-p}\right)\right]^\Gamma$.
\end{itemize}

Let $Z_1, Z_2$ be as in the Definition~\ref{def-quotienttype} of proper holomorphic maps of quotient type, 
i.e., $\Phi$ is a regular covering map from $\Omega_1\setminus Z_1$ to $\Omega_2\setminus Z_2$.
Let $U$ be an open set in $\Omega_2\setminus Z_2$ which is evenly covered 
by $\Phi$, and let $V$ be an open set of $\Omega_1\setminus Z_1$ which is mapped biholomorphically by $\Phi$ onto $U$. Then
the inverse image $\Phi^{-1}(U)$ is the  finite disjoint union $\bigcup_{\sigma \in \Gamma} \sigma V$. Therefore if
$f\in L^p(\Omega_2)$ is supported in $U$, then $\Phi^\sharp f$ is supported in $\bigcup_{\sigma \in \Gamma} \sigma V$, and we have
\begin{align*}
\norm{\Phi^\sharp f}^p_{L^p\left(\Omega_1, \abs{\det \Phi'}^{2-p}\right)}&=\int_{\Omega_1} \abs{f\circ \Phi \cdot \det \Phi'}^{p} \abs{\det \Phi'}^{2-p} dV =\sum_{\sigma \in \Gamma}\int_{\sigma V} \abs{f\circ \Phi}^p \abs{\det \Phi'}^2 dV\\
&= \sum_{\sigma\in \Gamma }\int_{U} \abs{f}^p dV  = \abs{\Gamma}\cdot\norm{f}_{L^p(\Omega_2)}^p,		\end{align*}
where we have used the change of variables formula applied to the biholomorphic map $\Phi|_{\sigma V}$
along with the fact that the real Jacobian determinant of the map $\Phi$ is equal to $\abs{\det\Phi'}^2$.	

For a general $f\in L^p(\Omega_2)$, modify the proof as follows. There is clearly a   collection of pairwise disjoint open sets  $\{U_j\}_{j\in J}$ in $\Omega_2\setminus Z_2$ such that
each $U_j$ is evenly covered by $\Phi$ and $\Omega_2 \setminus \bigcup_{j\in J} U_j$  has measure zero. Set $f_j = f\cdot\chi_j$,  where $\chi_j$ is the 
indicator function of $U_j$, so that 
each $f_j\in L^p(\Omega_2)$ and $f=\sum_{j}f_j$. Also, the functions $\Phi^\sharp f_j$ have pairwise disjoint supports in $\Omega_1$. Therefore
we have 
\[
\norm{\Phi^\sharp f}^p_{L^p\left(\Omega_1, \abs{\det \Phi'}^{2-p}\right)}= \sum_{j\in J} \norm{\Phi^\sharp f_j}^p_{L^p\left(\Omega_1, \abs{\det \Phi'}^{2-p}\right)}
= \abs{\Gamma} \sum_{j\in J} \norm{f_j}_{L^p(\Omega_2)}^p  = \abs{\Gamma}  \norm{f}_{L^p(\Omega_2)}^p.
\]
 
To complete the proof, we need to show that 
$\Phi^\sharp$ is surjective in both \eqref{eq-phisharp1} and \eqref{eq-phisharp2}. Let $g\in \left[L^p\left(\Omega_1, \abs{\det \Phi'}^{2-p}\right)\right]^{\Gamma} $. Let $\{U_j\}_{j\in J}$ be as in 
the previous paragraph, and set $g_j= g\cdot \chi_{\Phi^{-1}(U_j)}$, where $\chi_{\Phi^{-1}(U_j)}$ is the indicator function of $\Phi^{-1}(U_j)$.
Notice that $g_j\in \left[L^p\left(\Omega_1, \abs{\det \Phi'}^{2-p}\right)\right]^{\Gamma}$. Let $V_j\subset \Phi^{-1}(U_j)$ be such that 
$\Phi$ maps $V_j$ biholomorphically to $U_j$, and let $\Psi: U_j\to V_j$ be the local inverse of $\Phi$ onto $V_j$. Define
\begin{equation}
\label{eq-fj}
f_j = \Psi^\sharp (g_j).
\end{equation}

We claim that $f_j$ is defined independently of the choice of $V_j$. Indeed, any other choice is of the form $\sigma V_j$ for some $\sigma\in \Gamma$  and the
corresponding local inverse is $\sigma\circ \Psi$. But we have
\[
(\sigma\circ \Psi)^\sharp g_j = \Psi^\sharp \circ \sigma^\sharp g_j
= \Psi^\sharp g_j  =f_j,
\]
where we have used the fact that $ \sigma^\sharp g_j=g_j$ since
since $g_j\in \left[L^p\left(\Omega_1, \abs{\det \Phi'}^{2-p}\right)\right]^{\Gamma}$.

Now we define $f=\sum_{j}f_j$. Notice that the $f_j$ have pairwise disjoint support, and it is easily checked that $\Phi^\sharp f =g$. This establishes that \eqref{eq-phisharp1} is a homothetic isomorphism.

It is clear that if $f$ is holomorphic on $\Omega_2$, then $\Phi^\sharp f$ is holomorphic on $\Omega_1$, therefore, $\Phi^\sharp$ maps $A^p(\Omega_2)$ 
into $\left[A^p\left(\Omega_1, \abs{\det \Phi'}^{2-p}\right)\right]^{\Gamma}$. Now, in the argument in the previous paragraph showing that the image
$\Phi^\sharp(L^p(\Omega_2))$ is $\left[L^p\left(\Omega_1, \abs{\det \Phi'}^{2-p}\right)\right]^{\Gamma}$, local definition of the inverse map \eqref{eq-fj} shows
that if $g\in \left[A^p\left(\Omega_1, \abs{\det \Phi'}^{2-p}\right)\right]^{\Gamma}$, then the $f$ constructed by this procedure in holomorphic, and therefore 
lies in $A^p(\Omega_2)$. This completes the proof of the proposition.
\end{proof}
\subsection{Bell transformation law for quotient maps} The following is a refinement (for the class of proper holomorphic maps of quotient type)  of a  classic result of Bell   (see \cite[Theorem~1]{bellduke}, \cite[Equation~2.2]{belltransactions}).

\begin{prop}\label{prop-l2}
	Let $\Omega_1,\Omega_2$ be domains in $\cx^n$  and let $\Phi:\Omega_1\to \Omega_2$ be a proper holomorphic map of quotient type with group $\Gamma\subset \mathrm{Aut}(\Omega_1)$. Then the following diagram commutes:
	\begin{equation}\label{diag-l2}	
	\begin{tikzcd}
	L^2(\Omega_2)\arrow[r,"{\Phi}^\sharp"] \arrow[r,swap,"\cong"] \arrow[d, "\bm{B}_{\Omega_2}"] 
	&\left[L^2(\Omega_1)\right]^\Gamma \arrow[d,"\bm{B}_{\Omega_1}"] 
	\\
	{A}^2(\Omega_2)
	\arrow[r,"{\Phi}^\sharp"] \arrow[r,swap,"\cong"]
	& \left[{A}^2(\Omega_1)\right]^\Gamma
	\end{tikzcd}\end{equation}
\end{prop}

In order to prove the proposition, we need the following simple lemma, which shows that the Bergman projection interacts well with the action of automorphisms:
\begin{lem}
	\label{lem-inv}
	Let $\Omega\subset \cx^n$ be a domain, and let $\bm{B}_\Omega$ be its Bergman projection operator.
	\begin{enumerate}
		\item If $\sigma\in \mathrm{Aut}(\Omega)$ is a biholomorphic automorphism, then 
		\begin{equation}\label{eq-bomega1}
		\bm{B}_{\Omega}\circ \sigma^\sharp = \sigma^\sharp \circ \bm{B}_{\Omega}.
		\end{equation}
		\item If $G\subset \Aut(\Omega)$ is a group of biholomorphic automorphisms, then
		$\bm{B}_{\Omega}$ restricts to the orthogonal projection operator from $[L^2(\Omega)]^G$ 
		onto $[A^2(\Omega)]^G$.
	\end{enumerate}
\end{lem}
\begin{proof}
    For (1), note that  $\sigma^\sharp$ is a unitary operator  on $L^2(\Omega)$ and $\bm{B}_{\Omega}$ is an orthogonal projection on $L^2(\Omega)$, therefore the unitarily similar operator $Q=\sigma^\sharp \circ \bm{B}_{\Omega}\circ (\sigma^\sharp)^{-1}$ is also an orthogonal projection.  Since $ \sigma^\sharp$ (and therefore its inverse)
    leaves $A^2(\Omega)$ invariant, it follows that the range of $Q$ is  $A^2(\Omega)$. Therefore $Q=\bm{B}_{\Omega}$.
	
    For (2), let $f\in [L^2(\Omega)]^G$. Then using \eqref{eq-bomega1}, we have for     $\sigma\in G$
	\[ \sigma^\sharp (\bm{B}_{\Omega}f) = \bm{B}_{\Omega}(\sigma^\sharp (f)) = \bm{B}_{\Omega}f,\]
	which shows that $\bm{B}_{\Omega}f\in [A^2(\Omega)]^G$, so that $\bm{B}_{\Omega}$ maps the $G$-invariant functions $[L^2(\Omega)]^G$ into the $G$-invariant holomorphic functions $[A^2(\Omega)]^G$. Since $\bm{B}_\Omega$ restricts to the identity on $[A^2(\Omega)]^G$, 
	it follows that the range of $\bm{B}_{\Omega}$ is $[A^2(\Omega)]^G$.  Observe that
	\[ \ker\left(\bm{B}_{\Omega}|_{[L^2(\Omega)]^G}\right)\subseteq \ker \bm{B}_{\Omega}= A^2(\Omega)^\perp \subseteq \left([A^2(\Omega)]^G\right)^\perp, \]
	which shows that kernel of the  restriction of $\bm{B}_{\Omega}$ to $[L^2(\Omega)]^G$ is  orthogonal to its range, and therefore an orthogonal projection.
	\end{proof}

\begin{proof}[Proof of Proposition~\ref{prop-l2}]
By Proposition~\ref{prop-homothety},  the $\Phi^\sharp$ represented by the top (resp. bottom) horizontal arrow
is a homothetic isomorphism from  the Hilbert space  $	{L}^2(\Omega_2)$ 
(resp. $	{A}^2(\Omega_2)$) onto the Hilbert space $[{L}^2(\Omega_1)]^\Gamma $ (resp. $[{A}^2(\Omega_1)]^\Gamma $). Therefore, $\Phi^\sharp$ preserves angles and in particular  orthogonality. Now consider the map $ P: [{L}^2(\Omega_1)]^\Gamma  \to[{A}^2(\Omega_1)]^\Gamma  $
defined by
\begin{equation}\label{eq-Pdef}		 
P= \Phi^\sharp \circ \bm{B}_{\Omega_2}\circ (\Phi^\sharp)^{-1},
\end{equation}
which, being a composition of continuous linear maps, is a continuous linear mapping of Hilbert spaces. Notice that
\[
P^2 = \Phi^\sharp \circ \bm{B}_{\Omega_2}\circ (\Phi^\sharp)^{-1}\circ  \Phi^\sharp \circ \bm{B}_{\Omega_2}\circ (\Phi^\sharp)^{-1}= \Phi^\sharp \circ \bm{B}_{\Omega_2}\circ (\Phi^\sharp)^{-1}=P, 
\]
so $P$ is a projection in $[{L}^2(\Omega_1)]^\Gamma$, with range contained in  $[{A}^2(\Omega_1)]^\Gamma$. 	Since $(\Phi^\sharp)^{-1}$ and $\Phi^\sharp|_{A^2(\Omega)}$ are isomorphisms, and $\bm{B}_{\Omega_2}$ is surjective, it follows that $P$ is a projection onto $[{A}^2(\Omega_1)]^\Gamma$.  We claim that $P$ is in fact the \emph{orthogonal} projection on to $[{A}^2(\Omega_1)]^\Gamma$,
i.e., the kernel of $P$ is $\left([{A}^2(\Omega_1)]^\Gamma\right)^\perp$, the orthogonal complement of $[{A}^2(\Omega_1)]^\Gamma$ in $[{L}^2(\Omega_1)]^\Gamma$. Since  in formula \eqref{eq-Pdef}, the maps $(\Phi^\sharp)^{-1}$   and $\Phi^\sharp$ are isomorphisms, it follows that $f\in \ker P$ if and only if $(\Phi^\sharp)^{-1}f\in \ker \bm{B}_{\Omega_2}$. But $\ker \bm{B}_{\Omega_2}= A^2(\Omega_2)^\perp$, since the Bergman projection is orthogonal. It follows that $\ker P = \Phi^\sharp(A^2(\Omega_2)^\perp)$. Notice that $\Phi^\sharp$, being a homothetic isomorphism of Hilbert spaces, preserves orthogonality,
and maps $A^2(\Omega_2)$ to $[{A}^2(\Omega_1)]^\Gamma$ isomorphically, therefore $\Phi^\sharp( (A^2(\Omega_2))^\perp)= \left([{A}^2(\Omega_1)]^\Gamma\right)^\perp$, which establishes the claim.

Therefore we have shown that $P=\Phi^\sharp \circ \bm{B}_{\Omega_2}\circ (\Phi^\sharp)^{-1}$ is the orthogonal projection from $ [{L}^2(\Omega_1)]^\Gamma$
to the subspace   $[{A}^2(\Omega_1)]^\Gamma$.  To complete the proof, we only need to show that the restriction of the Bergman projection
$\bm{B}_{\Omega_1}$ to the $\Gamma$-invariant subspace  $[{L}^2(\Omega_1)]^\Gamma$ is also the orthogonal projection from $[{L}^2(\Omega_1)]^\Gamma$ 
onto 
$[{A}^2(\Omega_1)]^\Gamma$.  But this follows from Lemma~\ref{lem-inv} above.

Thus $P=\bm{B}_{\Omega_1}|_{[L^2(\Omega_1)]^\Gamma}$, and the commutativity of \eqref{diag-l2} follows.
\end{proof}

\subsection{Transformation of the Bergman projection in $L^p$-spaces} The following result will be our main tool on studying $L^p$-regularity of the Bergman projection in the domain $\Uu$:

\begin{thm}\label{thm-transformation}	
Let $\Omega_1,\Omega_2$ be bounded domains in $\cx^n$, let $\Phi:\Omega_1\to \Omega_2$ be a proper holomorphic map of quotient type with group $\Gamma\subset \mathrm{Aut}(\Omega_1)$.  Let $p\geq 1$. The  following two assertions are equivalent:
\begin{enumerate}
\item  \label{statement-a} The Bergman projection $\bm{B}_{\Omega_2}$ gives rise to a bounded operator mapping $$L^p(\Omega_2) \to A^p(\Omega_2).$$
\item  \label{statement-b} The Bergman projection $\bm{B}_{\Omega_1}$ gives rise to a bounded operator  mapping $$\left[L^p\left(\Omega_1, \abs{\det \Phi'}^{2-p}\right)\right]^{\Gamma} \to \left[A^p\left(\Omega_1, \abs{\det \Phi'}^{2-p}\right)\right]^{\Gamma}.$$		
\end{enumerate}

If one of the conditions \eqref{statement-a} or \eqref{statement-b} holds (and therefore both hold), then the following diagram commutes, where $\bm{B}_{\Omega_j}, j=1,2$ denote the extension by 
	continuity of the Bergman projections:
	\begin{equation} 
	\label{diag-bounded} 	
	\begin{tikzcd}			
	{L}^p(\Omega_2)
	\arrow[r,"{\Phi}^\sharp"] \arrow[r,swap,"\cong"] \arrow[d, "\bm{B}_{\Omega_2}"] 
	& \left[L^p\left(\Omega_1, \abs{\det \Phi'}^{2-p}\right)\right]^{\Gamma} \arrow[d,"\bm{B}_{\Omega_1}"] 
	\\
	{A}^p(\Omega_2)
	\arrow[r,"{\Phi}^\sharp"] \arrow[r,swap,"\cong"]
	& \left[A^p\left(\Omega_1, \abs{\det \Phi'}^{2-p}\right)\right]^{\Gamma}
	\end{tikzcd}
	\end{equation}
\end{thm}

\begin{rem}
Statement \eqref{statement-a} in Theorem~\ref{thm-transformation} means  the following:  the restriction of the Bergman projection to a dense subspace of $L^p(\Omega_2)$ given by
\begin{equation*}
\bm{B}_{\Omega_2}:L^2(\Omega_2)\cap L^p(\Omega_2) \to A^2(\Omega_2)
\end{equation*}
is bounded in the $L^p$-norm, i.e., there is a $C>0$ such that for all $f\in L^2(\Omega_2)\cap L^p(\Omega_2)$, 
\begin{equation*}
\norm{\bm{B}_{\Omega_2} f}_{L^p(\Omega_2)} \leq C \norm{f}_{L^p(\Omega_2)}.
\end{equation*} 
By continuity $\bm{B}_{\Omega_2}$ extends to a bounded linear operator from $L^p(\Omega_2)$ to $A^p(\Omega_2)$.
\smallskip
	
Similarly, statement \eqref{statement-b} means the following:  the restriction of the Bergman projection to the dense subspace of $\left[L^p\left(\Omega_1, \abs{\det \Phi'}^{2-p}\right)\right]^{\Gamma}$ given by
\begin{equation*} 
\bm{B}_{\Omega_1}: [L^2(\Omega_1)]^\Gamma \cap \left[L^p\left(\Omega_1, \abs{\det \Phi'}^{2-p}\right)\right]^{\Gamma} \to A^2(\Omega_1)
\end{equation*}
is bounded in the $L^p\left(\Omega_1, \abs{\det \Phi'}^{2-p}\right)$-norm, i.e., there is a $C>0$ such that for all $f\in [L^2(\Omega_1)]^\Gamma \cap \left[L^p\left(\Omega_1, \abs{\det \Phi'}^{2-p}\right)\right]^{\Gamma}$,
\begin{equation*}
\norm{\bm{B}_{\Omega_1}f}_{L^p\left(\Omega_1, \abs{\det \Phi'}^{2-p}\right)}\leq C \norm{f}_{L^p\left(\Omega_1, \abs{\det \Phi'}^{2-p}\right)}.
\end{equation*}
We now see by Lemma~\ref{lem-inv} that 
\begin{equation*} 
\bm{B}_{\Omega_1} \left(\left[L^p\left(\Omega_1, \abs{\det \Phi'}^{2-p}\right)\right]^{\Gamma}\right) \subseteq \left[A^p\left(\Omega_1, \abs{\det \Phi'}^{2-p}\right)\right]^{\Gamma},
\end{equation*} 
where we have used continuity to extend the operator.  
\end{rem}

\begin{proof}
Proposition~\ref{prop-homothety} says that $\Phi^\sharp$ is a homothetic isomorphism, mapping $$L^p(\Omega_2) \to \left[L^p(\Omega_1, \abs{\det \Phi'}^{2-p})\right]^\Gamma,$$ and that it restricts to a homothetic isomorphism on the holomorphic subspaces.  Similarly, $(\Phi^\sharp)^{-1}$ has the same properties with the domains and ranges switched.

First assume statement (2). From the diagram \eqref{diag-l2}, we write
\begin{equation}\label{eq-star} 
\bm{B}_{\Omega_2}=  (\Phi^\sharp)^{-1}\circ \bm{B}_{\Omega_1}\circ\Phi^\sharp.
\end{equation}

By hypothesis, $\bm{B}_{\Omega_1}$ is a bounded linear operator mapping $$\left[L^p(\Omega_1, \abs{\det \Phi'}^{2-p})\right]^\Gamma \to \left[A^p(\Omega_1, \abs{\det \Phi'}^{2-p})\right]^\Gamma.$$  Consequently, this composition maps $L^p(\Omega_2)$ boundedly into $A^p(\Omega_2)$, giving statement (1).  A similar argument shows that (1) implies (2).

For the commutativity of the diagram, rewrite \eqref{eq-star} and see that on the subspace $L^p(\Omega_2)\cap L^2(\Omega_2)$ of $L^p(\Omega_2)$ we have the relation 
\begin{equation}\label{eq-doublestar}
	\Phi^\sharp\circ\bm{B}_{\Omega_2}=   \bm{B}_{\Omega_1}\circ\Phi^\sharp.
\end{equation}
Using the hypothesis (for the $\bm{B}_{\Omega_j}$) and Proposition~\ref{prop-homothety} (for $\Phi^\sharp$) we see that each of the four maps
in the diagram \eqref{diag-bounded} extends to the  respective domain in that diagram and is continuous. By continuity, \eqref{eq-doublestar} continues to hold for the extended maps. This shows that the diagram \eqref{diag-bounded} is commutative.
\end{proof}

\begin{rem}
Diagram~\eqref{diag-l2} is a special case of diagram~\eqref{diag-bounded} for $p=2$. 
\end{rem}


\section{The unboundedness of the Bergman projection on \texorpdfstring{$\Uu$}{U}}\label{sec-unbounded}
Using Proposition~\ref{prop-unboundedness} below and the results of Section
\ref{sec-transformation}, we will prove in this section the following partial form of Theorem~\ref{thm-main}:
\begin{prop}
\label{prop-unboundedness}
The Bergman projection is not bounded in $L^p(\Uu)$, provided
\begin{equation}
    \label{eq-p}p \geq \frac{2\kappa(\Uu)}{\kappa(\Uu)-1},
\end{equation}
where $\kappa(\Uu)\in \N$ is the complexity of the domain $\Uu$, as defined in subsection~\ref{sec-mainresult}.
\end{prop}

\subsection{Reinhardt domains}

Recall some elementary facts about holomorphic function theory on Reinhardt domains (which are always assumed to be centered at the origin).
 Let $\Omega \subset \mathbb{C}^n$ be a Reinhardt domain. Every holomorphic function $f\in \mathcal{O}(\Omega)$ admits a unique \emph{Laurent expansion} 
\begin{align}\label{eq-laurent}
	f = \sum_{\alpha\in (\mathbb{Z}^n)^\dagger} a_\alpha(f) \varphi_\alpha, 
\end{align}
where for $\alpha\in (\mathbb{Z}^n)^\dagger$, $\varphi_\alpha(z)$ is the Laurent monomial $z^\alpha$ as in \eqref{eq-monomial}, and 
where $a_\alpha(f) \in \cb$  is the $\alpha$-th Laurent coefficient.  The Laurent series of $f$ converges absolutely and uniformly to $f$ on every compact subset of $\Omega$. 

When $f$ lies in the Bergman space $A^2(\Omega)$, we can say more about the series \eqref{eq-laurent}: it is actually an orthogonal series converging in the Hilbert space
$A^2(\Omega)$, and the family of monomials
\begin{equation}
	\label{eq-onb}
	{\left\{\dfrac{\varphi_\alpha}{\norm{\varphi_\alpha}_{L^2}}:  \varphi_\alpha\in L^2(\Omega)\right\}}
\end{equation}
forms an orthonormal basis of $A^2(\Omega)$. In particular, if $f\in A^2(\Omega)$ then the Laurent series \eqref{eq-laurent} can have $a_\alpha(f){\neq}0$ only when $\varphi_\alpha\in L^2(\Omega)$.  It is possible to generalize some of these results to the spaces $A^p(\Omega)$; see \cite{ChEdMc19}.

\subsection{A criterion for unboundedness of the Bergman projection}
We now give an easily checkable condition which shows $L^p$-Bergman unboundedness on any Reinhardt domain.

\begin{lem}\label{lem-unboundedness}
Let $\Omega$ be a bounded Reinhardt domain in $\cx^n$, and let $p\geq 2$. Suppose that there is a multi-index $\beta\in (\mathbb{Z}^n)^\dagger$ such that 
\begin{equation} 
\varphi_\beta\in L^2(\Omega)\setminus L^p(\Omega).
\end{equation}
Then the Bergman projection $\bm{B}_\Omega$ fails to map $L^p(\Omega) \to L^p(\Omega)$.
\end{lem}
\begin{proof}

Define subsets $\mathcal{J}_\beta, \mathcal{K}_\beta \subset \{1,2,\cdots,n \}$ with 
\begin{align*}
\mathcal{J}_\beta = \{\,j : \beta_j \ge 0 \}, \qquad \qquad  \mathcal{K}_\beta = \{\,k : \beta_k < 0 \},
\end{align*}
and let 
\[
f(w)= \prod_{j\in \mathcal{J}_\beta} w_j^{\beta_j} \times \prod_{k\in \mathcal{K}_\beta  } \left(\overline{w_k}\right)^{-\beta_k}.
\]
Then $f$ is a bounded function on $\Omega$, and therefore $f \in L^p(\Omega)$. 
We now show that $\bm{B}_\Omega f = C \varphi_\beta$ for some $C{\neq}0$. Since $\varphi_\beta\not\in L^p(\Omega)$, this will show that $\bm{B}_\Omega$ fails to map the element $f\in L^p(\Omega)$ to a function in $L^p(\Omega)$. This will imply that $ \bm{B}_\Omega$  is not bounded in the $L^p$-norm, since if it were so, it would extend to a map from 
the dense subspace $L^p(\Omega)\cap L^2(\Omega)$ to the whole of $L^p(\Omega)$. 

Let $\gamma=(\abs{\beta_1}, \dots, \abs{\beta_n})\in \mathbb{N}^n$ be the multi-index obtained by replacing each entry of $\beta$ by its absolute value. Write the polar form of $w$ as
$w=\rho(w)\gdot \exp(i\theta)$ for a $\theta\in \rl^n$. Then
$
f(w) = \rho(w)^\gamma e^{i \beta  \theta},
$
and for $\alpha\in (\mathbb{Z}^n)^\dagger$,  we  have
\[ 
\varphi_\alpha(w)=w^\alpha= \left(\rho(w)\gdot \exp(i\theta) \right)^\alpha= \rho(w)^\alpha e^{i \alpha  \theta}. 
\]
Further, denote by $\abs{\Omega}=\{\rho(z):z\in \Omega\}\subset \rl^n$ the Reinhardt shadow of $\Omega$,  and let $\mathbb{T}^n$ be the unit torus of $n$ dimensions.
Then, for each $\alpha\in \mathbb{Z}^n $, we have
\begin{align*}
	\ipr{\bm{B}_\Omega f, \varphi_\alpha}_{A^2(\Omega)}&= \ipr{f,\varphi_\alpha}_{L^2(\Omega)}= \int_\Omega f \,\ol{\varphi_\alpha}dV\\
	&=\int_\Omega r^\gamma e^{i\beta\theta}\cdot r^\alpha e^{-i\alpha\theta} dV
	= \int_{\abs{\Omega}} r^{\gamma+\alpha} r^\one dr \times \int_{\mathbb{T}^n} e^{i((\beta-\alpha)\cdot\theta)} d\theta\\
	& \begin{cases} =0 & \text{if }\alpha\not = \beta \\
	> 0 & \text{if } \alpha = \beta.
	\end{cases}
\end{align*}
Since \eqref{eq-onb} is an orthonormal basis of $A^2(\Omega)$ it follows that
all the Fourier coefficients of $\bm{B}_\Omega f$ with respect this basis vanish, except the $\beta$-th coefficient, which is nonzero. Therefore, 
$\bm{B}_\Omega f = C \varphi_\beta \notin L^p(\Omega),
$ for some constant $C{\neq}0$.
\end{proof}

\subsection{Preliminaries}\label{sec-prelim}

In this section (and the following section~\ref{sec-bounded}), we will use the following default notation and conventions: 
\begin{enumerate}
    \item $B\in M_n(\Z)$ is a matrix
such that the domain $\Uu$ is represented as in \eqref{eq-udef2}, 
\item $B$ satisfies the properties \eqref{eq-B}, which is not a loss of generality by Proposition~\ref{prop-uopen}.
\item We have $A=\adj B$.
\end{enumerate}
Observe that then by Theorem~\ref{thm-geometry}, the monomial map $\Phi_A$ is a proper holomorphic map of quotient type. 

We note the following computation:
\begin{prop} The upper bound in \eqref{eq-bounds}
is given by \begin{equation}
    \label{eq-upperbound}\frac{2\kappa(\Uu)}{\kappa(\Uu)-1}=\min_{j} \frac{2\cdot \one\cdot a_j}{\one\cdot a_j- \gcd(a_j)},
\end{equation}
where, as usual, $a_j$ is the $j$-th column of the matrix $A$.
\end{prop}
\begin{proof}Recall the definition of the projective 
height function as in  \eqref{eq-height}. Then we have, by projective 
invariance, for each $1\leq j \leq n$:
\[\mathsf{h}(B^{-1}e_j)=\mathsf{h}(\det B\cdot B^{-1}e_j)=\mathsf{h}_j(\adj B\cdot e_j)= \mathsf{h}(A e_j)=\mathsf{h}(a_j).  \]
Since $A\succ 0$, it follows that $a_j$ is a vector of non-negative 
integers, and the vector $\frac{1}{\gcd(a_j)}a_j$ is such that its entries are coprime non-negative 
integers. Therefore
\[\mathsf{h}(B^{-1}e_j)=\mathsf{h}(a_j)=\mathsf{h}\left(\frac{1}{\gcd(a_j)}a_j\right) = \sum_{k=1}^n \frac{\abs{a_j^k}}{\gcd(a_j)}=\frac{\one\cdot a_j}{\gcd(a_j)}.  \]
So 
\[ \kappa(\Uu)=\max_{j} \mathsf{h}(B^{-1}e_j)= \max_{j}\frac{\one\cdot a_j}{\gcd(a_j)}.\]

Since the function $x\mapsto \dfrac{2x}{x-1}$ is strictly decreasing for $x>1$, we have
\[ \frac{2\kappa(\Uu)}{\kappa(\Uu)-1}=\frac{2\cdot \max_{j}\dfrac{\one\cdot a_j}{\gcd(a_j)}}{\max_{j}\dfrac{\one\cdot a_j}{\gcd(a_j)}-1}=\min_j\frac{\dfrac{2\cdot\one\cdot a_j}{\gcd(a_j)}}{\dfrac{\one\cdot a_j}{\gcd(a_j)}-1}=\min_{j} \frac{2\cdot \one\cdot a_j}{\one\cdot a_j- \gcd(a_j)}.\]
\end{proof}

\subsection{$p$-allowable multi-indices}
Let $\beta\in (\Z^n)^\dagger$ be a multi-index. We say that $\beta$ is \emph{$p$-allowable} on a Reinhardt domain $\Omega$ if we have that the monomial
$\varphi_\beta\in L^p(\Omega)$, i.e., 
 \[ \int_\Omega \abs{\varphi_\beta}^p dV<\infty.\]
 We denote the collection of $p$-allowable multi-indices on $\Omega$ by $\mathscr{S}_p(\Omega)$. We first compute 
 the $p$-allowable multi-indices on the domain $\Uu$. Recall that the conventions introduced in Section~\ref{sec-prelim} are in force.

\begin{prop}
\label{prop-pallow}
Let $\beta\in (\Z^n)^\dagger$  and $p>0$. Then $\beta\in \mathscr{S}_p(\Uu)$ if and only if 
\begin{equation}
    \label{eq-pallow}(p\cdot\beta+ 2\cdot\one)A\succ 0.
\end{equation}
\end{prop}

\begin{proof} By Theorem~\ref{thm-geometry}, $\Phi_A:\D^n_\slb\to \Uu$ is a proper 
holomorphic map of quotient type with $\det A$ sheets. 
Therefore, by  Proposition~\ref{prop-homothety}(in particular \eqref{eq-phisharphomo}) we see that
\begin{align*}
    \int_{\Uu} \abs{\varphi_\beta}^p dV &= \frac{1}{\abs{\Gamma}}\int_{{\D^n_{\mathsf{L}(B)}}}\abs{\Phi_A^\sharp(\varphi_\beta)}dV\\
    &=\frac{1}{{\det A}}\int_{{\D^n_{\mathsf{L}(B)}}}\abs{\varphi_\beta\circ \Phi_A}^p\abs{\det\Phi_A'}^2 dV\\
    &= \frac{1}{{\det A}}\int_{{\D^n_{\mathsf{L}(B)}}} \abs{z^{\beta\cdot A}}^p\abs{{\det A}\cdot z^{\one\cdot A-\one}}^2 dV(z)\\&\text{ using \eqref{eq-prop6} and \eqref{eq-det2}}\\
    &=(2\pi)^n{\det A}\int_{(0,1)^n}r^{p\cdot \beta A}\cdot{r^{2\cdot\one\cdot A-2\one }}\cdot r^{\one} dr\\
    &=(2\pi)^n{\det A}\int_{(0,1)^n}r^{p\cdot \beta A+ 2\cdot\one\cdot A- \one}dr\\
    &= (2\pi)^n{\det A}\prod_{j=1}^n \int_0^1 r_j^{p \cdot \beta a_j + 2 \cdot \one a_j -1} dr_j,
\end{align*} 
where $a_j \in \zb^n$ is the $j$-th column of $A$. It is clear that $\int_{\Uu} \abs{\varphi_\beta}^p dV<\infty$ if and only if $p \cdot \beta a_j + 2 \cdot \one a_j>0$ for each $1\leq j \leq n$. This completes the proof.
\end{proof}
\begin{rem}The above proposition can be seen as a special case of \cite[Lemma~2.2.1]{zwonekhab}.
\end{rem}

We now consider the important case $p=2$, so that $\mathscr{S}_2(\Omega)$ corresponds to the monomials in the Bergman space. For a matrix $A\in M_n(\Z)$, none of whose columns are zero,
denote by $\gcd(a_j)$ the greatest common divisor of the entries in the $j$-th column of $A$. We then let
\begin{equation}
\label{eq-gofa}
g(A)= \left(\gcd(a_1),\dots, \gcd(a_n)\right)\in (\Z^n)^\dagger 
\end{equation}
be the integer row vector whose $j$-th entry is the greatest common divisor of the $j$-th column of $A$.
\begin{prop}\label{prop-2allow}
  \begin{enumerate} 
    \item Let $\beta\in (\Z^n)^\dagger$. Then $\beta\in \mathscr{S}_2(\Uu)$ if and only if 
    \begin{equation}
        \label{eq-2allow}
        (\beta+ \one)A\succeq g(A).
    \end{equation}
\item For $1\leq j \leq n$ let $\Pi_j$ be the integer hypersurface determined 
by equality in the $j$-th entry of \eqref{eq-2allow}, that is
\begin{equation}
    \label{eq-Pij}\Pi_j=\left\{ \beta \in (\Z^n)^\dagger :(\beta+ \one)a_j= \sum_{k=1}^n (\beta_k+1)a^k_j=\gcd(a_j) \right\}.
\end{equation}
Then we have
\[ \Pi_j\cap \mathscr{S}_2(\Uu){\neq}\emptyset.\] 
\end{enumerate}

\end{prop}
The following Lemma will be needed in the proof of part (2) of the proposition.
\begin{lem} \label{lem-part2} For positive integers $m,n$, with $n\geq m$,
let $P$ be an $m\times n$ integer matrix of rank $m$,
and  let $q\in (\Z^n)^\dagger$. There there is an $x\in (\Z^m)^\dagger$ such that $xP\succeq q$.
\end{lem}
\begin{proof}
Let $\mathcal{O}=\{x\in (\rl^m)^\dagger:xP\succeq q \} .$ Then $\mathcal{O}$ is an unbounded convex set, so at least one of the coordinates $x_1,\dots, x_m$ is unbounded on $\mathcal{O}$. Rename the coordinates so that $x_1$ is unbounded. It follows that the projection
\[ \{x_1\in\rl: (x_1,\dots, x_m)\in \mathcal{O}\} \]
of $\mathcal{O}$
on the coordinate axis $x_1$ is an unbounded convex set, and therefore a ray or the whole of $\rl$.
For an integer $N$ let $C_N=\{x_1=N\}\cap \mathcal{O}$, which is naturally thought of as a subset of $(\rl^{m-1})^\dagger$.
Therefore   either there is an $N_1$ such that $C_N$ is nonempty if $N\geq N_1$
or there is an $N_2$ such that $C_N$ is  nonempty if $N\leq N_2$. Assuming the former, we see that the sets $C_N$ are convex subsets of $(\rl^{m-1})^\dagger$ and similar to each other, i.e. they are dilations of the same set. As the size of
each $C_N$ becomes infinite as $N\to \infty$, for large $N$, the set $C_N$ contains cubes of arbitrarily large size, where a cube is a product of intervals of the same size in each coordinate. As soon as $C_N$ contains a cube 
of side $(1+\epsilon)$ for some $\epsilon>0$, we see that there is a point
$M\in (\Z^{m-1})^\dagger$ that belongs to $C_N$. It follows that the point $(N,M)\in \Z\times(\Z^{m-1})^\dagger$ belongs to $\mathcal{O}$.

\end{proof}
\begin{proof}[Proof of Proposition~\ref{prop-2allow}]
	(1) If $p=2$, the condition \eqref{eq-pallow} becomes $2(\beta+\one)A\succ 0$, which is equivalent to 
	\begin{equation}
	    \label{eq-2allow2}  (\beta+\one)A\succ 0.
	\end{equation}
	Now the $j$-th entry of the row vector on the left of the above equation is given by
	$(\beta+\one)a_j=\sum_{k=1}^n (\beta_k+1)a^k_j , $
	which is a positive integer divisible by 
	$\gcd(a_j)=\gcd(a^1_j,\dots, a^n_j).$ It follows that \eqref{eq-2allow2} holds if and only if
		\[(\beta+\one)a_j \succeq \gcd(a_j),\]
	which is precisely the content of \eqref{eq-2allow}.
	
	(2) Fix $1\leq j \leq n$. By the Euclidean algorithm, $\Pi_j {\neq}\emptyset$.  Choose $y \in \Pi_j$. Define a $\Z$-module homomorphism  $\phi: (\Z^n)^\dagger \to \Z$ by setting $\phi(x)=xa_j$, i.e.,
	\[ \phi(x_1, \ldots, x_n) =  \sum_{k=1}^n x_k a^k_j.\]
	We  then see that $\Pi_j = y + \ker\phi$. Since  $\Z$ is a principal ideal domain,   $\ker \phi$ is a free   $\Z$-submodule of $(\Z^n)^\dagger$  of  rank $\leq n$ (see \cite[ Theorem 4, Chapter 12 (page 460)]{dummit}). Moreover as $\phi$ is surjective,  the quotient $\Z$-module $(\Z^n)^\dagger/\ker\phi$  is isomorphic to $\Z$. It can be seen, by tensoring with $\mathbb Q$ for example, that  the rank of $\ker \phi$ is $n-1$.  Let
	$D$ be an $(n-1)\times n$ integer matrix whose rows are a $\Z$-basis of $\ker \phi$. Then the map $f:(\Z^{n-1})^\dagger\to (\Z^n)^\dagger$ given by
	\[ f(t)=y+tD\]
	is a parametrization of $\Pi_j$, i.e., it is one-to-one and its range is precisely $\Pi_j$. To complete the proof of the  result, it is sufficient to show that 
	$f^{-1}(\mathscr{S}_2(\Uu))$ is a non-empty subset of $(\Z^{n-1})^\dagger$. Notice now that an integer vector $t\in  f^{-1}(\mathscr{S}_2(\Uu))$, 
	i.e. $f(t)\in \mathscr{S}_2(\Uu)\cap \Pi_j$ if and only if 
	\[ (y+tD+\one)A\succeq g(A),\]
	i.e.
	\[tDA \succeq g(A)-(y+\one)A.\]
	By Lemma~\ref{lem-part2}, there is an integer vector $t\in (\Z^{n-1})^\dagger$ that satisfies the above system of inequalities. This concludes the proof of part (2)

\end{proof}

\subsection{Proof of Proposition~\ref{prop-unboundedness}}
Recall that $p$ satisfies  \eqref{eq-p}. Now by \eqref{eq-upperbound},
\[ p\geq \frac{2\kappa(\Uu)}{\kappa(\Uu)-1}=\min_{j} \frac{2\cdot \one\cdot a_j}{\one\cdot a_j- \gcd(a_j)}, \]
so that there is a 
$J$ with $1\leq J \leq n$ such that 
\begin{equation}
    \label{eq-p2}p \geq \frac{2\cdot\one\cdot a_J}{\one\cdot a_J- \gcd(a_J)}.
\end{equation}
By part 2 of proposition~\ref{prop-2allow}, there is an $\beta\in (\Z^n)^\dagger$
which lies in $\mathscr{S}_2(\Uu)\cap \Pi_J$. Such a $\beta$ satisfies:
\begin{equation}
    \label{eq-beta1}(   \beta+\one)a_J = \gcd(a_J),
\end{equation}
and also
\begin{equation}
    \label{eq-beta2}
    (\beta+\one)A \succeq g(A),
\end{equation}
with $g(A)$ as in \eqref{eq-gofa}.
By construction, $\beta\in \mathscr{S}_2(\Uu)$. We now claim that $\beta\not \in \mathscr{S}_p(\Uu)$. By Lemma~\ref{lem-unboundedness} this shows that the Bergman projection is not bounded in $L^p(\Uu)$. To establish the claim we note that the $J$-th entry of the row vector
$(p\cdot\beta+2\cdot \one)A$ is
\begin{align}
    (p\cdot\beta+2\cdot \one)a_J&= 
    p\cdot(\beta_k+\one)\cdot a_J+(2-p) \cdot\one\cdot a_J\nonumber\\
    &= p \cdot \gcd(a_J)+(2-p)\cdot \one\cdot a_J\nonumber\\
    &=p\cdot \left(\gcd(a_J)- \one \cdot a_J \right)+2\cdot \one\cdot a_J\label{eq-beta3}
\end{align}
where in the second line we have used \eqref{eq-beta1}.
Thanks to the inequality \eqref{eq-p2} it follows that the 
quantity in \eqref{eq-beta3} is not positive. It follows by 
Proposition~\ref{prop-pallow} that $\beta$ is not in $L^p(\Uu)$, which establishes the claim and completes the proof.

\section{Boundedness of the Bergman projection}\label{sec-bounded}
In this section, we obtain the following part of
Theorem~\ref{thm-main}:
\begin{prop}
\label{prop-bounded}
Let
\begin{equation}
    \label{eq-hypbounded}
     2\leq p < \frac{2\kappa(\Uu)}{\kappa(\Uu)-1},
\end{equation}
then the Bergman projection is bounded on $L^p(\Uu)$.
\end{prop}
We begin by recalling the following fact, which will be the main ``hard analysis" ingredient of the proof:

\begin{prop}\label{P:LpBergmanPolydisc}
The Bergman projection on the polydisc $\mathbb{D}^n$ is a bounded operator $\bm{B}_{\mathbb{D}^n}:L^p(\mathbb{D}^n) \to A^p(\mathbb{D}^n)$ for all $1<p<\infty$.
\end{prop}
\begin{proof} For the polydisc $\mathbb{D}^n$, the Bergman projection has the well-known integral representation 
	\[ \bm{B}_{\mathbb{D}^n} f (z)= \int_{\mathbb{D}^n} K(z,w) f(w)  dV(w),\quad f\in L^2(\mathbb{D}^n),\]
	where $K$ is the Bergman kernel of the polydisc, which is easily shown to be given by the well-known formula \eqref{eq-bpdsk}.

The  case $n=1$ of Propsition~\ref{P:LpBergmanPolydisc}  is by now a staple result in Bergman theory, going back to \cite{ZahJud64}, where it was proved using the $L^p$-boundedness of a  Calderon-Zygmund singular integral operator. Another approach, based on Schur's test for $L^p$-boundedness of an integral operator,  was used in \cite{rudin}. An
alternative  proof of the main estimate needed in this method can be found in \cite{axler} and in the monograph \cite{durenbergman}.

   Since $\mathbb{D}^n$ is a product domain, the theorem in higher dimensions follows from a textbook application of Fubini's theorem to the case $n=1$.
\end{proof}

\subsection{Two lemmas} The following two simple lemmas will be used to deduce monomially weighted estimates starting from 
Proposition~\ref{P:LpBergmanPolydisc}:

\begin{lem}\label{lem-polydisc}
	Let $1\leq p<\infty$, let $n\geq 1$, and let $\gamma\in (\rl^n)^\dagger$ be such that 
	$\gamma\succ -2\cdot\one$. Then
	there is a $C>0$ such that for any $f\in A^p(\mathbb{D}^n)$ we have
	\begin{equation}\label{eq-firstestimate}
	\int_{\mathbb{D}^n} \abs{f}^p \rho^\gamma \,dV \leq C  \int_{\mathbb{D}^n} \abs{f}^p dV,
	\end{equation}
	where as usual, $\rho(z)^\gamma= \prod_{j=1}^n \abs{z_j}^{\gamma_j}$.
	\end{lem}

\begin{proof} Throughout this proof, $C$ will denote a constant that depends only on $p$ and $\gamma$.  The actual value of $C$ may change from line to line.
	
Proceed by induction on the dimension $n$. First consider the base case $n=1$.  We have, by the Bergman inequality (cf. \cite[Theorem~1]{durenbergman})  that there is a $C>0$ such that
	\[ \sup_{\abs{z}< \frac{1}{2}} \abs{f(z)}\leq C \norm{f}_{L^p(\mathbb{D})} \]
	for all $f\in A^p(\mathbb{D})$. Therefore, for $f\in A^p(\mathbb{D})$ we have the estimate
	\begin{equation}
	\label{eq-est1}
	\int_{\abs{z}<\frac{1}{2}}\abs{f(z)}^p \abs{z}^\gamma dV(z)\leq  \sup_{\abs{z}< \frac{1}{2}} \abs{f(z)}^p \cdot \int_{\abs{z}<\frac{1}{2}} \abs{z}^\gamma dV(z) < C \cdot\norm{f}_{L^p(\mathbb{D})}^p,
	\end{equation}
	where we have used the fact that since $\gamma>-2$ we have
	\[\int_{\abs{z}<\frac{1}{2}} \abs{z}^\gamma dV(z)=2\pi\int_0^{\frac{1}{2}} r^{\gamma+1}dr = \frac{2\pi}{\gamma+2}\left(\frac{1}{2}\right)^{\gamma+2}<\infty. \]
	On the other hand, 
	\begin{equation}
	\label{eq-est2}
	\int_{\frac{1}{2}\leq \abs{z}<1} \abs{f(z)}^p \abs{z}^\gamma dV(z)\leq C \norm{f}_{L^p(\mathbb{D})}^p,
	\end{equation}
	where we have used the fact that
\[ \sup_{\frac{1}{2}\leq \abs{z}<1} \abs{z}^\gamma <\infty. \]	

	Adding \eqref{eq-est1} and \eqref{eq-est2}, the estimate ~\eqref{eq-firstestimate} follows in the case $n=1$.
	
	For the general case, assume the result established in $n-1$ dimensions. Write the coordinates of $\cx^n$ as $z=(z',z_{n})\in \cx^{n-1}\times \cx$, and $\gamma=(\gamma',\gamma_n)\in \rl^{n-1}\times \rl$.
 Then, using Fubini's theorem
	\begin{align*}
	\int_{\mathbb{D}^n}\abs{f}^p \rho^{\gamma} dV& = \int_{\mathbb{D}^{n-1}}\rho(z')^{\gamma'}\left(\int_{\mathbb{D}} \abs{f(z',z_n)}^p\abs{z_n}^{\gamma_n}dV(z_n)\right)dV(z')\\
	&\leq C  \int_{\mathbb{D}^{n-1}}\rho(z')^{\gamma'}\left(\int_{\mathbb{D}} \abs{f(z',z_n)}^p dV(z_n)\right)dV(z')\\
	&\leq C \int_{\mathbb{D}}\left(\int_{\mathbb{D}^{n-1}}\abs{f(z',z_n)}^p \rho(z')^{\gamma'} dV(z')\right)dV(z_n)\\
	&\leq C  \int_{\mathbb{D}}\left(\int_{\mathbb{D}^{n-1}}\abs{f(z',z_n)}^p dV(z')\right)dV(z_n)\\
	&=C 	\int_{\mathbb{D}^n}\abs{f(z',z_n)}^p  dV(z',z_n),
	\end{align*}
	which proves the result.
\end{proof}

\begin{lem} \label{lem-secondestimate}	Let $1\leq p<\infty$, let $n\geq 1$ and let $\lambda\in \mathbb{N}^n$ be a multi-index of non-negative integers.
Then  there is a $C>0$ such that for all $f\in A^p(\mathbb{D}^n)$ we have the estimate:
	\begin{equation}
	\label{eq-polydisc2}\int_{\mathbb{D}^n}\abs{f}^p dV \leq C \int_{\mathbb{D}^n}\abs{\varphi_\lambda f}^p dV,
	\end{equation}
where $\varphi_\lambda(z)=z^\lambda$ is as in \eqref{eq-monomial}.
\end{lem}
\begin{proof}
	We need only to prove the case in which $\lambda=(1,0,\dots,0)$, so that $\varphi_\lambda(z)=z_1$. Once this special case is established, the general result follows by repeatedly applying it and 
	permuting the coordinates.
	
	In what follows, $C$ will denote some positive constant that depends only on $p$ and $\lambda$, where the actual value of $C$ may change from line to line.
	First consider the one dimensional case, so that we have to show that for a holomorphic function $f$ on the disc we have
	\[\int_{\mathbb{D}}\abs{f(z)}^p dV(z) \leq C \int_{\mathbb{D}}\abs{z f(z)}^p dV(z),  \]
	where the left hand side is assumed to be finite (and therefore the right hand side is finite.) First note that we obviously have
	\begin{equation}
	\label{eq-est3}
	\int_{\frac{1}{2}\leq \abs{z}<1}\abs{f(z)}^p dV(z) \leq 2^p  \int_{\frac{1}{2}\leq \abs{z}<1}\abs{z f(z)}^p dV(z).
	\end{equation}
	On the other hand, if $\abs{z}=\frac{1}{2}$, we have
\begin{align*}
    \abs{f(z)}^p &= 2^p \abs{z f(z)}^p\\
	&\leq 2^p \sup_{\abs{w}\leq \frac{1}{2}} \abs{w f(w)}^p& \text{(maximum principle)}\\
	&\leq C \int_{\mathbb{D}}\abs{w f(w)}^p dV(w) &\text{(Bergman's inequality)}
	\end{align*}
	The maximum principle now implies
	\[  \sup_{\abs{z}\leq \frac{1}{2}} \abs{f(z)}^p \leq C \int_{\mathbb{D}}\abs{z f(z)}^p dV(z),  \]
	so that we have
	\begin{equation}\label{eq-est4}
	\int_{\abs{z}\leq \frac{1}{2}} \abs{f(z)}^p \leq C \int_{\mathbb{D}}\abs{z f(z)}^p dV(z). 
	\end{equation}
	Combining \eqref{eq-est3} and \eqref{eq-est4} the result follows for $n=1$.
	
	For the higher-dimensional case, denote the coordinates of $\cx^n$ as $(z_1, z')$ where $z'=(z_2,\dots, z_n)$. Then for $f\in A^p(\mathbb{D}^n)$ we have
	\begin{align*}
	\int_{\mathbb{D}^n} \abs{f(z_1, z')}^pdV(z_1, z')& = \int_{\mathbb{D}^{n-1}} \left(\int_{\mathbb{D}} \abs{f(z_1,z')}^p dV(z_1)\right)dV(z')\\
	&\leq C \int_{\mathbb{D}^{n-1}} \left(\int_{\mathbb{D}} \abs{z_1f(z_1,z')}^p dV(z_1)\right)dV(z')\\
	&= C\int_{\mathbb{D}^{n}} \abs{z_1f(z_1,z')}^p dV(z_1,z').
	\end{align*}	
\end{proof}

\subsection{Determination of  $\Gamma$-invariant subspaces }
In this subsection and the next subsection \ref{sec-bddproof}, we use 
the notation established in subsection~\ref{sec-prelim} above, so that $B$ and $A$ have the same meaning as there. The group $\Gamma$ is as in 
Theorem~\ref{thm-geometry}: the deck-transformation group associated 
to the map $\Phi_A$.
We will now determine the $\Gamma$-invariant subspaces of holomorphic functions, in the sense of \eqref{eq-gammainv}:
\begin{prop}
 Let $\alpha\in (\Z^n)^\dagger$. Then the monomial $\varphi_\alpha(z)=z^\alpha$ belongs to the space
$\left[\mathcal{O}\left((\D^*)^n\right)\right]^\Gamma$ of $\Gamma$-invariant holomorphic functions on $(\D^*)^n$
 if and only if there is a
$\beta\in (\Z^n)^\dagger$ such that 
\[\alpha=\beta A-\one. \]
\end{prop}
\begin{proof}
Recall that, by definition, the monomial  $\varphi_\alpha$ is invariant under the group action of $\Gamma$ if and only if $\sigma_\nu^\sharp(\varphi_\alpha)= (\varphi_\alpha \circ \sigma_\nu) \det \sigma_\nu' = \varphi_\alpha$
  for all $\sigma_\nu \in \Gamma$, where for $\nu \in \Z^n$, the automorphism $\sigma_\nu\in \Gamma$ is as in \eqref{eq-sigmanu}. . 
 Denoting the rows of $A^{-1}$  be denoted by  by $c^1, \ldots, c^n$, for $\nu \in \Z^n$, by the linearity of $\sigma_\nu$ 
\[\det \sigma_\nu' = \det \sigma_\nu=\prod_{j=1}^n e^{2 \pi i c^j \nu}
= e^{2 \pi i \sum_{j=1}^n c^j \nu}= e^{2 \pi i (\one A^{-1}) \nu}. \]
Also, for $z \in (\D^*)^n$ and $ \nu \in \Z^n$, 
\[\varphi_\alpha \circ \sigma_\nu (z) = \varphi_\alpha( \exp( 2 \pi i A^{-1} \nu) \gdot z ) = (\exp( 2 \pi i A^{-1} \nu) \gdot z )^\alpha
= e^{2 \pi i \alpha A^{-1} \nu} z^\alpha.
\]
Therefore
\begin{equation}
    \label{eq-sigmasharp}
    \sigma_\nu^\sharp(\varphi_\alpha)(z)= e^{2 \pi i \alpha A^{-1} \nu}\cdot   z^\alpha \cdot e^{2 \pi i \one A^{-1} \nu}=e^{2\pi i (\alpha +\one)A^{-1}\nu}\cdot z^\alpha=e^{2\pi i (\alpha +\one)A^{-1}\nu}\cdot \varphi_\alpha(z).
\end{equation}

Hence 
    $\varphi_\alpha \in  \left[\mathcal{O}\left({\D^n_{\mathsf{L}(B)}}\right)\right]^\Gamma$ if and only if $\sigma_\nu^\sharp(\varphi_\alpha)=\varphi_\alpha$ for all $\nu\in \Z^n$,
i.e.,     
    \[e^{2\pi i (\alpha +\one)A^{-1}\nu}\cdot z^\alpha = z^\alpha\]
    for all $z \in (\D^*)^n$ and $ \nu \in \Z^n, $ i.e., if and only if   $(\alpha + \one) A^{-1} \nu \in \Z \text{ for all  } \nu \in \Z^n. $
 
  Now if there is $\beta \in (\Z^n)^\dagger$ such that $\alpha = \beta A - \one$, then clearly,   $(\alpha + \one) A^{-1} \nu =\beta\nu\in \Z$   for all $\nu \in \Z^n$. Conversely, assume that  $(\alpha + \one) A^{-1} \nu \in \Z$   for all $\nu \in \Z^n$ and let $e_1,\dots, e_n$ be the standard basis of $\Z^n$. Then the $j$-th column of $(\alpha + \one) A^{-1}= (\alpha + \one) A^{-1}I$ (where $I$ is the $n\times n$ identity matrix) is  $(\alpha+\one)A^{-1}e_j$ which is therefore in $\Z$. Therefore we have $(\alpha + \one) A^{-1}=\beta \in (\Z^n)^\dagger$. It follows that
$\alpha = \beta A - \one$ 
as desired.  \end{proof}

\begin{cor}\label{cor-invfun}
Let $f\in \left[\mathcal{O}(\D^n)\right]^\Gamma$. Then there is an
$h\in \mathcal{O}(\D^n)$ such that 
\[f(z)= z^{g(A)-\one}\cdot h(z), \]
with $g(A)$ as in \eqref{eq-gofa}.
\end{cor}
\begin{proof} Let 
$\displaystyle{ f(z)=\sum_{\alpha\succeq 0}a_\alpha z^\alpha}$,
be the Taylor expansion of $f$. If $f$ is $\Gamma$ invariant, then we claim that for
each $\alpha$ such that $a_\alpha{\neq}0$, we have that $z^\alpha$ is $\Gamma$-invariant. Indeed using 
\eqref{eq-sigmasharp} we have
\[f=\sigma_\nu^\sharp f =\sigma_\nu^\sharp\left(\sum_{\alpha\succeq 0}a_\alpha\varphi_\alpha  \right)=\sum_{\alpha\succeq 0}
a_\alpha e^{2\pi i (\alpha A^{-1}+\one)\nu}\cdot \varphi_\alpha,\]
comparing this with the Taylor expansion of $f$ and equating coefficients the claim follows. Therefore, the Taylor expansion of $f$ is of the form
\begin{equation}
    \label{eq-looks}
    f(z)= \sum_{\substack{\beta\in (\Z^n)^\dagger\\ \beta A\succeq\one}} a_{\beta A-\one} z^{\beta A-\one}.
\end{equation}
Notice that $\displaystyle{z^{\beta A-\one} =\prod_{j=1}^n z_j^{\beta a_j-1} } $. The integer
$\beta a_j=\sum_{k=1}^n\beta_k a_j^k$ which occurs in the exponent is divisible by $\gcd(a_j)$. Further since $\sum_{k=1}^n\beta_k a_j^k\geq 1$ it follows that $\beta a_j\geq \gcd(a_j).$ It follows that the monomial $z^{\beta A-\one}$ is of the form
$z^{g(A)-\one}\cdot z^\gamma$, where $\gamma\succeq 0$. The corollary follows from \eqref{eq-looks}.

\end{proof}
\subsection{Proof of Proposition~\ref{prop-bounded}}\label{sec-bddproof}
By Theorem~\ref{thm-geometry}, $\Phi_A: \D^n_\slb\to\Uu $ is a proper holomorphic map of quotient type
with group $\Gamma$.
Therefore, thanks to Theorem~\ref{thm-transformation}, the Bergman projection
\[ \bm{B}_{\Uu}: L^p(\Uu)\to L^p(\Uu)\]
is bounded if and only if 
\begin{equation}
    \label{eq-bdnslb} \bm{B}_{\D^n_\slb} : \left[ L^p(\D^n_\slb , \abs{\det \Phi_A'}^{2-p}) \right]^{\Gamma} \to \left[ A^p(\D^n_\slb , \abs{\det \Phi_A'}^{2-p}) \right]^{\Gamma}
\end{equation}
is bounded. We therefore show that the  map \eqref{eq-bdnslb} is bounded if $2\leq p < \frac{2\kappa(\Uu)}{\kappa(\Uu)-1}$.

To do this we express the map of \eqref{eq-bdnslb} as a composition of three maps, and show that each of these three maps is continuous:

\begin{enumerate}[wide]
    \item We first show that we have an inclusion of Banach spaces
    \[ \left[L^p(\D^n_\slb , \abs{\det \Phi_A'}^{2-p}) \right]^{\Gamma} \subseteq  \left[L^p(\D^n_\slb)\right]^\Gamma,\] and that the inclusion map 
\begin{equation}
\label{eq-inclusion}
    \imath:  \left[ L^p(\D^n , \abs{\det \Phi_A'}^{2-p}) \right]^{\Gamma} \xhookrightarrow{} \left[L^p(\D^n)\right]^\Gamma
\end{equation}
is bounded.

It is clearly sufficient to show that there is a continuous inclusion of the full space  $ L^p(\D^n_\slb , \abs{\det \Phi_A'}^{2-p}) $  into $L^p(\D^n_\slb)$. 
For
$z\in \D^n$ we have for $p\geq 2$
\[ \abs{\det \Phi_A'(z)}^{2-p}= \abs{\det A z^{\one A-\one}}^{2-p}=\abs{\det A}^{2-p}\rho(z)^{(\one\cdot A-\one)(2-p)} \geq \abs{\det A}^{2-p}, \]
where we use the fact that the exponent of $\rho(z)$ is nonpositive, since $A$ has nonnegative integer entries and
$p\geq 2$.
Thus for any $p \geq 2$ and any $f \in L^p(\D^n_\slb , \abs{\det \Phi_A'}^{2-p}) $ we have
\begin{equation}
\label{eq-intinclusion}
    \int_{\D^n_\slb} \abs{f}^p dV \leq \frac{1}{\abs{ \det A}^{2-p}} \int_{\D^n_\slb} \abs{f}^p \abs{ \det \Phi_A'}^{2-p} dV,
\end{equation}
which proves the inclusion and its continuity.

\item Let $P_\Gamma= \left.\bm{B}_{\D^n_\slb}\right|_{\left[ L^p(\D^n_\slb) \right]^{\Gamma} }$ be the restriction of the Bergman projection operator  to the $\Gamma$-invariant functions. We claim that the operator of Banach spaces
\begin{equation}
    \label{eq-pgamma}
    P_\Gamma: \left[ L^p(\D^n_\slb) \right]^{\Gamma} \to \left[ A^p(\D^n_\slb) \right]^{\Gamma}
\end{equation}
is bounded. Indeed, we have that $L^p(\D^n_\slb)=L^p(\D^n)$ (since $\D^n_\slb$ is obtained from $\D^n$ by removing an analytic set $Z$, which is  of measure zero) and $ A^p(\D^n_\slb)= A^p(\D^n)$, in the sense that the analytic set $Z$ is a removable singularity of functions integrable in the $p$-th power, $p\geq 2$ (see \cite[p. 687 ]{belltransactions}). By Proposition~\ref{P:LpBergmanPolydisc}, the Bergman projection
maps $L^p(\D^n)$ to  $A^p(\D^n)$ boundedly.  Finally, by part (2) of Lemma~\ref{lem-inv}, the Bergman projection maps $\Gamma$-invariant functions to $\Gamma$-invariant functions. The boundedness of $P_\Gamma$ follows.
\item Finally, we show that there is an inclusion $\left[A^p(\D^n_\slb)\right]^\Gamma \subset \left[ A^p (\D^n_\slb, \abs{ \det \Phi_A'}^{2-p}) \right]^\Gamma$, and the inclusion map so determined is bounded.
As noted above $A^p(\D^n_\slb)=A^p(\D^n)$, so it will suffice to show that there is a continuous inclusion
\begin{equation}
    \label{eq-j} \jmath :\left[A^p(\D^n)\right]^\Gamma \xhookrightarrow{} \left[ A^p (\D^n, \abs{ \det \Phi_A'}^{2-p}) \right]^\Gamma.
\end{equation}
Let $f \in \left[A^p(\D^n)\right]^\Gamma$, so by Corollary~\ref{cor-invfun}, there exists $h \in \mathcal{O}(\D^n)$ such that 
\begin{equation}
    \label{eq-f}f(z)=z^{g(A)-\one}h(z).
\end{equation}
In fact, $h\in A^p(\D^n)$, since 
\[\int_{\D^n}\abs{h(z)}^pdV =\int_{\rho(z)\preceq \frac{1}{2}\one}\abs{h(z)}^pdV + \int_{\frac{1}{2}\one\prec \rho(z)\prec\one }\abs{h(z)}^pdV,\]
where the first of the two integrals is clearly finite and the second integral is
\[ =\int_{\frac{1}{2}\one\prec \rho(z)\prec\one }\abs{\frac{f(z)}{z^{g(A)-\one}}}^pdV\leq \left(\prod_{j=1}^n 2^{\gcd(a_j)-1} \right)^p\cdot\int_{\frac{1}{2}\one\prec \rho(z)\prec\one} \abs{f(z)}^pdV <\infty.  \]

Now we have that
\begin{align}
    \norm{f}_{A^p \left(\D^n, \abs{ \det \Phi_A'}^{2-p}\right)}^p&= \int_{\D^n} \abs{f(z)}^p \abs{ \det \Phi_A'(z)}^{2-p}dV(z)\nonumber\\ &= \int_{\D^n}\abs{h(z)}^p \abs{z^{g(A)-\one}}^p \abs{ \det A\cdot{z^{\one A-\one}}}^{2-p} dV(z) \nonumber\\
    & \text{ using \eqref{eq-f} and \eqref{eq-det2}} \nonumber \\
    &=\abs{ \det A }^{2-p}\cdot\int_{\D^n}\abs{h(z)}^p\abs{z^{p(g(A)-\one) + (2-p)(\one A - \one)}} dV(z)\nonumber\\
    &=\abs{ \det A }^{2-p}\cdot\int_{\D^n}\abs{h}^p\rho^{ p(g(A)-\one\cdot A) + 2\cdot\one\cdot A - 2\cdot\one}dV \label{eq-fnorm}
\end{align}
Combining hypothesis \eqref{eq-hypbounded} and \eqref{eq-upperbound},
\[p< \frac{2\kappa(\Uu)}{\kappa(\Uu)-1}= \min_{j} \frac{2\cdot \one\cdot a_j}{\one\cdot a_j- \gcd(a_j)},\] so for 
each $1\leq j \leq n$,
\[p< \frac{2\cdot\one\cdot a_j }{\one\cdot a_j-\gcd(a_j)}. \]
Therefore the $j$-th component of the exponent of $\rho$ in \eqref{eq-fnorm} is 
\[p(\gcd(a_j)-\one a_j) + 2\cdot\one\cdot a_j - 2> \left(\frac{2\cdot\one\cdot a_j }{\one\cdot a_j-\gcd(a_j)} \right)(\gcd(a_j)-\one\cdot a_j) + 2\one a_j - 2=-2,\]
where we have used the obvious fact that
$\gcd(a_j)-\one\cdot a_j \leq 0$. Therefore we have
\[ p(g(A)-\one A) + 2\cdot\one A - 2\cdot\one\succ - 2\cdot\one,\]
and so we have for constants $C_1, C_2$ independent of the function $f$:
\begin{align*}
    \eqref{eq-fnorm}&\leq C_1 \int_{\D^n}\abs{h}^pdV & \text{ by Lemma~\ref{lem-polydisc}}\\
    & \leq C_2 \int_{\D^n}\abs{z^{g(A)-\one}\cdot h(z)}^pdV(z) & \text{ by Lemma~\ref{lem-secondestimate}}\\
    &= C_2\int_{\D^n}\abs{f}^pdV.
\end{align*}
It follows that the inclusion $j$ of \eqref{eq-j} is continuous.
\end{enumerate}
Therefore, the map $\bm{B}_{\D^n_\slb}$ of Banach spaces in \eqref{eq-bdnslb} can be represented as a composition
\[\bm{B}_{\D^n_\slb}= \jmath\circ P_\Gamma \circ \imath\]
where $\imath$, $P_\Gamma$ and $\jmath$ are as in \eqref{eq-inclusion}, \eqref{eq-pgamma} and \eqref{eq-j} respectively
and each of which has already been shown to be continuous. The proof of Proposition~\ref{prop-bounded} is complete.
\section{Conclusion}
\subsection{End of proof of Theorem~\ref{thm-main}}
Recall that as a consequence of the self-adjointness of the Bergman projection on $L^2$, the set of $p$ for which the Bergman projection on  a domain is
$L^p$-bounded is Hölder-symmetric (see \cite{EdhMcN16, chakzeytuncu}), i.e., $p$ belongs 
to this set if and only if the conjugate index
$p'$ also belongs to it, where $\frac{1}{p'}+\frac{1}{p}=1$. Notice now that the
index conjugate to $\frac{2\kappa(\Uu)}{\kappa(\Uu)-1}$ is $\frac{2\kappa(\Uu)}{\kappa(\Uu)+1}$ . Now by combining Propositions~\ref{prop-bounded} and \ref{prop-unboundedness}, we see that for $p\geq 2$, the Bergman projection on $\Uu$ is bounded if and only if $p\in \left[2, \frac{2\kappa(\Uu)}{\kappa(\Uu)-1} \right).$
By Hölder symmetry, for $p\leq2$, the Bergman projection is bounded if and only if 
$p\in \left(\frac{2\kappa(\Uu)}{\kappa(\Uu)+1},2\right] $. The bounds \eqref{eq-bounds} and therefore Theorem~\ref{thm-main} is proved. 

\subsection{Comments and questions}
The methods used to prove Theorem~\ref{thm-main} are more general than the result itself, and apply to the Bergman projection on various quotient domains of
simple domains with known Bergman kernels, provided the quotient type proper holomorphic map is a monomial map. For example, we may deduce using a modification of our arguments, the range of $p$ for which the Bergman projection is bounded in $L^p$ on the domain 
\[\{(\abs{z_1}^2+\dots+\abs{z_{n-1}}^2)^\frac{k_1}{2}<\abs{z_n}^{k_2}<1\}\subset\cx^n, \]
where $k_1,k_2$ are positive integers. 

Thanks to  Theoren~\ref{thm-main}, the Bergman projection is no longer bounded in 
$L^p(\Uu)$ if 
$p\geq \frac{2\kappa(\Uu)}{\kappa(\Uu)-1}$. It is natural to ask if there is an alternate projection from $L^p(\Uu)$ to $A^p(\Uu)$ for such $p$. In the special case
$\Uu=H_{m/n}$, the generalized Hartogs triangle of \eqref{eq-genhart}, it is possible to construct for each $p\geq 2$ a 
\emph{sub-Bergman projection} which gives rise to a bounded projection on $L^p$. In a future work, we will study whether this generalizes to the general monomial polyhedron $\Uu$.

Finally, we would like to understand the precise geometric significance of the arithmetic complexity of $\Uu$ without reference to the representation in terms of the matrix $B$. Such a description will pave the way of generalizing the results of this paper to wider classes of domains.

\bibliographystyle{alpha}
\bibliography{monomial}
\end{document}